\documentclass[11pt,a4paper]{amsart}
\usepackage[margin=1in]{geometry}
\usepackage[utf8]{inputenc}
\usepackage{amsmath}
\usepackage{amsfonts}
\usepackage{xypic}
\usepackage{amsthm}
\usepackage{mathrsfs}
\usepackage{hyperref}
\usepackage{amssymb}
\usepackage[all]{xy}
\usepackage{color}

\numberwithin{equation}{section}

\newtheorem{theorem}{Theorem}[section]
\newtheorem{lemma}[theorem]{Lemma}
\newtheorem{proposition}[theorem]{Proposition}
\newtheorem{corollary}[theorem]{Corollary}

\theoremstyle{remark}
\newtheorem{remark}[theorem]{Remark}

\DeclareMathOperator{\Jac}{Jac}

\DeclareMathOperator{\Hom}{Hom}
\DeclareMathOperator{\Blow}{Blow}

\renewcommand{\det}{\mathrm{det}}
\newcommand{\st}{\mathrm{st}}
\newcommand{\ch}{\mathrm{ch}}
\newcommand{\Todd}{\mathrm{Todd}}

\newcommand{\Pic}{\mathrm{Pic}}
\newcommand{\Hilb}{\mathrm{Hilb}}
\newcommand{\supp}{\mathrm{supp}}
\newcommand{\End}{\mathrm{End}}
\newcommand{\Ext}{\mathrm{Ext}}

\newcommand{\id}{\mathbf{1}}
\newcommand{\rk}{\mathrm{rk}}

\newcommand{\GL}{\mathrm{GL}}
\newcommand{\SL}{\mathrm{SL}}
\newcommand{\Sp}{\mathrm{Sp}}

\newcommand{\Tot}{\mathrm{Tot}}
\newcommand{\tr}{\mathrm{tr} \,}

\newcommand{\sps}{\mathrm{str.ps}}

\newcommand{\Kum}{\mathrm{Kum}}
\newcommand{\smooth}{\mathrm{sm}}

\newcommand{\ten}{\mathrm{ten}}
\newcommand{\six}{\mathrm{six}}

\renewcommand{\a}{\mathfrak{a}}

\newcommand{\Ee}{\mathcal{E}}
\newcommand{\Ff}{\mathcal{F}}

\newcommand{\Ii}{\mathcal{I}}

\newcommand{\Ll}{\mathcal{L}}
\newcommand{\Mm}{\mathcal{M}}
\newcommand{\Nn}{\mathcal{N}}
\newcommand{\Oo}{\mathcal{O}}
\newcommand{\Pp}{\mathcal{P}}

\newcommand{\Tt}{\mathcal{T}}

\newcommand{\Vv}{\mathcal{V}}

\newcommand{\Xx}{\mathcal{X}}

\newcommand{\f}{\mathrm{f}}
\newcommand{\g}{\mathrm{g}}

\renewcommand{\H}{\mathrm{H}}

\newcommand{\M}{\mathrm{M}}
\newcommand{\N}{\mathrm{N}}

\renewcommand{\P}{\mathrm{P}}

\renewcommand{\v}{\mathrm{v}}

\newcommand{\aA}{\mathbf{A}}
\newcommand{\bB}{\mathbf{B}}
\newcommand{\cC}{\mathbf{C}}

\newcommand{\hH}{\mathbf{H}}

\newcommand{\mM}{\mathbf{M}}
\newcommand{\nN}{\mathbf{N}}

\newcommand{\sS}{\mathbf{S}}

\newcommand{\vV}{\mathbf{v}}

\newcommand{\xX}{\mathbf{X}}

\newcommand{\ol}[1]{\overline{#1}}

\newcommand{\wt}[1]{\widetilde{#1}}

\renewcommand{\AA}{\mathbb{A}}

\newcommand{\CC}{\mathbb{C}}

\newcommand{\ZZ}{\mathbb{Z}}
\newcommand{\LL}{\mathbb{L}}

\newcommand{\KK}{\mathbb{K}}

\newcommand{\PP}{\mathbb{P}}

\newcommand\Quotient[2]{
        \mathchoice
            {
                \text{\raise1ex\hbox{\thinspace $#1$}\Big{/} \lower1ex\hbox{$#2$} \thinspace}%
            }
            {
                #1\,/\,#2
            }
            {
                #1\,/\,#2
            }
            {
                #1\,/\,#2
            }
    }

\newcommand\GIT[2]{
        \mathchoice
            {
                \text{\raise1ex\hbox{\thinspace $#1$}\Big{/}\!\!\!\!\Big{/} \lower1ex\hbox{$#2$} \thinspace}%
            }
            {
                #1\,/\,#2
            }
            {
                #1\,/\,#2
            }
            {
                #1\,/\,#2
     a       }
    }

\newcommand{\map}[5]{\begin{array}{ccc}   #1  & \stackrel{#5}{\longrightarrow} &  #2  \\  #3 & \longmapsto & #4  \end{array}}

\newcommand{\morph}[6]{\begin{array}{cccc} #6: & #1  & \stackrel{#5}{\longrightarrow} &  #2  \\ & #3 & \longmapsto & #4  \end{array}}

\title[O'Grady spaces and symplectic resolution of Higgs moduli spaces]{\bf O'Grady spaces and symplectic resolution of moduli spaces of Higgs bundles}

\author[E. Franco]{Emilio Franco}
\address{Emilio Franco,
\newline\indent Escuela Superior T\'ecnica de Ingenieros Industriales, 
\newline\indent Universidad Polit\'ecnica de Madrid, 
\newline\indent C. de Jos\'e Guti\'errez Abascal, 2, 28006, Madrid, Espa\~na}
\email{emilio.franco@upm.es}

\thanks{The author has been supported by the Scientific Employment Stimulus program, funded by FCT (Portugal), by an FCT Investigator grant with fellowship reference CEECIND/04153/2017. He was previously supported by the project PTDC/MAT- GEO/2823/2014 also funded by FCT (Portugal) with national funds.}
\date{\today}

\begin{document}

\maketitle

\begin{abstract}
We describe here a degeneration of the symplectic desingularization of the moduli spaces of topologically trivial $\GL(2,\CC)$ and $\SL(2,\CC)$-Higgs bundles over a hyperelliptic curve, into O'Grady's ten and six dimensional exceptional examples of irreducible holomorphic symplectic manifolds. Most of this note is a survey of work on these degenerations by Donagi--Ein--Lazarsfeld, de Cataldo--Maulik--Shen and Felissetti--Mauri, although in certain cases we provide details that were missing the previous articles. 
\end{abstract}

\tableofcontents

\begin{flushright}
{\it Dedicated to Peter Newstead \\ on the occasion of his 80th birthday.}
\end{flushright}

\section{Introduction}

Peter's contribution to the theory of moduli spaces is invaluable. In this note we want to honor him by describing a beautiful relation between two important classes of moduli spaces.

\subsection{Background}

A compact irreducible holomorphic symplectic manifold is a smooth compact simply connected K\"ahler manifold with a non-degenerate holomorphic 2-form which is unique up to scaling. Thanks to S.-T. Yau's proof of Calabi’s conjecture \cite{yau1, yau2}, a compact irreducible holomorphic symplectic manifold is equivalent to a {\it compact hyperk\"ahler manifold}. These manifolds play a central role in mathematical physics as they are one of three building blocks \cite{beauville_IHSM} for the construction of Ricci-flat manifolds, along with abelian varieties and Calabi--Yau manifolds. The list of known examples of compact irreducible holomorphic symplectic manifolds is surprisingly short \cite{beauville_3} when compared with the other types of building blocks: only two series of examples in arbitrary dimension are known (Hilbert schemes of points on a K3 and generalized Kummer varieties as shown by Beauville in \cite{beauville_IHSM}), along with two exceptional examples in dimension ten and six obtaineded by O'Grady \cite{OGrady_1, OGrady_2}. It is remarkable that all these examples are associated to moduli spaces of sheaves over symplectic surfaces, as the later are equipped with Mukai's holomorphic symplectic form \cite{mukai1}. In particular, O'
Grady's exceptional examples appear after a symplectic resolution of singular moduli spaces of sheaves on symplectic surfaces. The existence of symplectic resolution of these singular moduli spaces has been studied in detail \cite{lehn&sorger, kaledin&lehn&sorger, kiem}, concluding that only the class of cases considered by O'Grady admit one, hence closing the door to the appearance of a new compact irreducible holomorphic symplectic manifold out of this method. 

When we study moduli space of pure dimension $1$ sheaves on a symplectic surface, it is possible to consider a morphism to a projective space using the fitting support. Beauville \cite{beauville_fibr} showed that this provides an integrable system, known as the {\it Beauville--Mukai system}. In fact, a surpring result of Matsushita \cite{matsushita} implies that the only fibrations one can consider on compact holomorphic symplectic manifolds are integrable systems.

\

Higgs bundles were introduced by Hitchin \cite{hitchin-self} and its moduli space has become an important object of study thanks to its rich geometry. Non-abelian Hodge theory \cite{hitchin-self, simpson1, simpson2, donaldson, corlette} shows that the moduli spaces of Higgs bundles and (projectively) flat connections are homeomorphic, what equips the underlying manifold with a hyperK\"ahler structure. Therefore, these moduli spaces constitute a great source of examples of {\it non-compact hyperK\"ahler varieties}, hence holomorphic symplectic. These moduli spaces are smooth when the rank and the degree are chosen to be coprime and singular in the non-coprime case, as showed by Simpson \cite{simpson2}. It is then natural to ask whether these spaces admit a symplectic desingularization, a question that was first addressed by Kiem and Yoo \cite{kiem&yoo} via the study of the stringy $E$-functions. Using the work of Bellamy and Schedler \cite{bellamy&schedler} on symplectic resolutions of character varieties, Tirelli \cite{tirelli} showed moduli spaces of Higgs bundles admit a symplectic resolution only when the rank and the genus of the base curve are both equal to $2$. In this case, the desingularization was first provided by Yoo \cite{yoo} applying the same techniques used by O'Grady \cite{OGrady_1, OGrady_2}, although after \cite{bellamy&schedler, tirelli} such symplectic resolution can be obtained by simply blowing-up the singular locus of the moduli space. 

Another important feature of the moduli space of Higgs bundles is that it carries a morphism \cite{hitchin_duke} to a vector space, turning it into an integrable system named the {\it Hitchin system}. Thanks to the spectral correspondence \cite{hitchin_duke, BNR, simpson2}, the moduli space of Higgs bundles can be described as an open subset of the moduli space of pure dimension $1$ sheaves on a certain ruled surface and one can interpret the Hitchin system in terms of the fibration obtained by considering the (fitting) support of these pure dimension $1$ sheaves.

\

As we have seen, both the Beauville--Mukai and the Hitchin systems can be obtained by considering the fitting support on moduli spaces of pure dimension $1$ sheaves on surfaces, the first one being compact and second one non-compact. These similarities become even more explicit with the construction of the Donagi--Ein--Lazarsfeld \cite{donagi&ein&lazarsfeld} degeneration of the Beauville--Mukai integrable system into the Hitchin integrable system by considering the normal cone degeneration of a curve inside a K3 surface. This has been generalized in several directions, de Cataldo, Maulik and Shen \cite{deCataldo&maulik&shen_2} constructed a similar degeneration working on abelian surfaces in their proof of the P=W conjecture for $\SL_n$, Sawon and Shen \cite{sawon&shen} worked over a certain del Pezzo surface in order to construct a degeneration into the $\Sp(2m,\CC)$-Hitchin system, and the author studied in \cite{franco} the appearance of these degeneration over any smooth surface, obtaining a degeneration into the moduli space of $L$-twisted Higgs bundles. Sawon \cite{sawon} wrote a beautiful survey on the subject.

\subsection{Summary}

In this note we provide a review of the previously mentioned degenerations and all the objects involved in these constructions: moduli spaces of sheaves in smooth surfaces, compact irreducible holomorphic symplectic manifolds, Higgs bundles, and the ruled surface associated to the canonical bundle of a curve. 

We focus in a particular degeneration where the exceptional examples of O'Grady in dimensions ten and six are are taken to be our compact (irreducible) holomorphic symplectic varieties, and where we consider symplectic resolutions of the moduli space of $\GL(2,\CC)$ and $\SL(2,\CC)$-Higgs bundles over a curve of genus $2$ as our non-compact holomorphic symplectic varieties. Felissetti and Mauri \cite{felissetti&mauri}, discussed the existence of such degenerations by considering, in the ten dimensional space, the degeneration provided by Donagi--Ein--Lazarsfeld \cite{donagi&ein&lazarsfeld}, while relying on the work of de Cataldo--Maulik--Shen \cite{deCataldo&maulik&shen_2} for the six dimensional example. Our main contribution consists on providing details that were missing in the construction of the degeneration of \cite{deCataldo&maulik&shen_2}.

\subsection{Structure}

In Section \ref{sc sheaves on surfaces} we study the properties of moduli spaces of sheaves over smooth projective surfaces, which is a class of objects that are crucial both in the theory of compact irreducible holomorphic symplectic manifolds (when the base surface is either a K3 or an abelian surface) or in the theory of Higgs bundle (that are intimately related to certain sheaves over ruled surface).

In Section \ref{sc OGrady} we review some aspects of the theory of compact irreducible holomorphic symplectic manifolds, focusing on the construction of O'Grady's exceptional examples in dimensions ten and six.

Fix a smooth projective curve $C$ and its canonical bundle $K$. We study in Section \ref{sc canonical ruled surface} some properties of the canonical ruled surface $\KK = \PP(\Oo_C \oplus K)$ as it  plays an important role in the theory of Higgs bundles, which we address in Section \ref{sc Higgs moduli spaces}. In the later, we study the moduli spaces of Higgs bundles for the groups $\GL_n$ and $\SL_n$, recalling in which cases they admit a symplectic resolution.

In Section \ref{sc degeneration of ten dimensional} we review the non-linear deformation described by Donagi, Ein and Lazarsfeld reproducing as well the degeneration of the holomorphic symplectic structure. These results provide a degeneration, as symplectic varieties, of the moduli space of sheaves over symplectic surfaces into the moduli space of $\GL_n$-Higgs bundles. In Section \ref{sc SL_n case} we describe the generalization to the $\SL_n$-case in great detail. 

Finally, in Section \ref{sc main result} we make use of all of the above to describe a degeneration of the O'Grady ten and six dimensional spaces into the symplectic resolution of, respectively, the Higgs moduli spaces for $\GL_2$ and $\SL_2$.

\subsubsection*{Acknowledgments}

The author wants to thank Peter Newstead for all his teachings and for sharing his passion about mathematics.

\section{Moduli spaces of sheaves on surfaces}
\label{sc sheaves on surfaces}

In this section we describe some properties of moduli spaces of sheaves in smooth projective surfaces, which will be the main objects in our study.

Consider a smooth projective surface $S$. We say that a coherent sheaf $\Ff$ on $S$ is of pure dimension $d$ when the dimension of its schematic support is $d$ and all subsheaves are supported on dimension $d$ subschemes. After equipping $S$ with a polarization $\H$, one can check that the corresponding Hilbert polynomial $P(\Ff,\H)$ has degree $d$. Setting the coefficient multiplied by $d!$ to be its $\H$-polarized rank $\rk(\Ff, \H)$, we say that $\Ff$ is {\it $\H$-stable} ({\it resp.} {\it $\H$-semistable}) if it is of pure dimension $d$ and every proper subsheaf $\Ff' \subset \Ff$ satisfies $P(\Ff',\H)/\rk(\Ff',\H) < P(\Ff,\H)/\rk(\Ff,\H)$ ({\it resp.} $P(\Ff',\H)/\rk(\Ff',\H) \leq P(\Ff,\H)/\rk(\Ff,\H)$) when $n \gg 0$. If $\Ff$ is semistable, one further say that it is polystable when it decomposes as a direct sum of stable sheaves. 

Gieseker \cite{gieseker} constructed the moduli space $\M_{S}^\H(p)$ of pure dimension $d$ $\H$-semistable sheaves with Hilbert polynomial $p$, a construction that was extended by Simpson \cite{simpson1} to any projective varity. We denote by $\M_S^H(p)^{\st}$ the locus $\H$-stable sheaves, and by $\M_S^H(p)^{\sps}$ the locus of strictly polystable sheaves.

In the case of a smooth surface $S$, the topological invariants of a sheaf $\Ff$ associated to the Hilbert polynomial are the $\H$-polarized rank $\rk(\Ff,H)$, the first Chern class $c_1(\Ff)$, and the Euler characteristic $\chi(\Ff)$ (which contains information on the previous invariants together with the second Chern class $c_2(\Ff)$). In order to simplify the computations, one introduces the {\it Mukai vector}, $v := \ch(\Ff) \cdot \sqrt{\Todd(S)}$. Observe that $v \in H^{2*}(S, \ZZ)$ and the cup product in cohomology provides a pairing on the $\ZZ$-module $H^{2*}(S, \ZZ)$, hence acquiring a lattice structure. The definition of the pairing is meant to satisfy 
\begin{equation} \label{eq v^2}
v^2 = \chi(\End(\Ff)).
\end{equation}
Since the Mukai vector contains all the information to reconstruct the Hilbert polynomial, we then abbreviate by $\M^{\H}_S(v)$ the moduli space of $\H$-semistable sheaves $S$ with Hilbert polynomial determined by the Mukai vector $v \in H^{2*}(S, \ZZ)$. It is always possible write the Mukai vector as the multiple of a primitive class,
$$
v = n \cdot v_0.
$$
When $v = v_0$ is itself primitive, all coherent sheaves with such Mukai vector are forcely stable, so $\M_S^\H(v_0)^\st = \M_S^\H(v_0)$. A sheaf $\Ff$ is simple if its automorphism group is given by scalar multiplication, in particular $\Hom_S(\Ff, \Ff) \cong \CC$ in that case. It can be shown that stable sheaves are simple by a standard argument.

Thanks to deformation theory, the tangent space to $\M_{S}^{\H}(v)$ at $\Ff$ corresponds with
\[
\Tt_\Ff \M_{S}^\H(v) = \Ext^1_{S}(\Ff, \Ff),
\]
and, by Serre duality, the associated cotangent space is
\[
\Tt^*_\Ff \M_{S}^\H(v) = \Ext^1_{S}(\Ff, \Ff \otimes K_S).  
\]

A {\it symplectic surface} is a surface $S$ endowed with a holomorphic symplectic form $\Omega_S$. This is only the case when $S$ is either a K3 or abelian surface. One can understand $\Omega_S$ as a non-vanishing section of the canonical bundle $K_S$ which is forcely trivial. Due to Serre duality, one has the identification
\[
\Hom_S(\Ff, \Ff) \cong \Ext_S^2(\Ff,\Ff)^*,
\]
and both spaces are one dimensional when $\Ff$ is simple, implying
\[
\dim_\CC \Ext^1_S(\Ff,\Ff) = v^2 + 2,
\]
as a consequence of \eqref{eq v^2}. This shows that the simple locus of $\M_S^\H(v)$ is smooth, and in particular $\M_S^\H(v_0)$ is itself smooth when $v_0$ is a primitive Mukai vector.  

Thanks to Serre duality, $H^2(S, \Oo_S)$ is dual to $H^0(S, K_S) \cong \CC$ in the case of symplectic surfaces. This was used by Mukai to construct a holomorphic symplectic form on $\M_{S}^\H(v)$.

\begin{theorem}[\cite{mukai1}]
Let $S$ be a symplectic surface. Then, the simple locus $\M_{S}^\H(v)$ is smooth and can be equipped with a holomorphic $2$-form $\Omega_\M$ defined by taking the trace of the Yoneda product composed with $\Omega_S$,
\[
\Ext^1_{S}(\Ff, \Ff) \wedge \Ext^1_{S}(\Ff, \Ff) \stackrel{\circ}{\longrightarrow} \Ext^2_{S}(\Ff, \Ff) \stackrel{\tr}{\longrightarrow}H^{2}(S,\Oo_{S}) \stackrel{\Omega_S(\cdot)}{\longrightarrow} \CC.
\]
Furthermore, $\Omega_\M$ is non-degenrate on the smooth locus of $\M_{S}^\H(v)$, defining a symplectic form there.
\end{theorem}

A {\it Poisson surface} is a surface $S$ equipped with Poisson bi-vector $\Theta_S$, a non-zero section of $\bigwedge^2 \Tt_S \cong K_S^{-1}$. The above construction was generalized independently by Bottaccin \cite{bottacin_1} and Markman \cite{markman} to the case of Poisson surfaces.

\begin{theorem}[\cite{bottacin_1} and \cite{markman}] \label{tm Bottaccin-Markman Poisson str}
Let $S$ be is equipped with a Poisson bi-vector $\Theta_S$. Then, one can define a (closed although possibly degenerate) Poisson structure $\Theta_\M$ on the smooth locus of $\M_{S}^\H(v)$ by taking the trace of the Yoneda product composed with $\Theta_S$,
\[
\Ext^1_{S}(\Ff, \Ff\otimes K_S) \wedge \Ext^1_{S}(\Ff, \Ff \otimes K_S) \stackrel{\circ}{\longrightarrow} \Ext^2_{S}(\Ff, \Ff\otimes K_S^2) \stackrel{\tr}{\longrightarrow}H^{2}(S,K_S^2) \stackrel{\langle \cdot , \Theta_S \rangle}{\longrightarrow} H^{2}(S,K_S) \cong \CC.  
\]
\end{theorem}

On a smooth surface $S$, every pure dimension $1$ sheaf $\Ff$ has a locally free resolution of length $2$. This allow us to define the {\it fitting support} of a sheaf $\Ff$ as the determinant associated to this resolution, which we denote by $\supp(\Ff)$. Observe that  
\[
c_1(\Ff) = \left [ \supp(\Ff) \right ]. 
\]
We then see that the Mukai vector of a pure dimension $1$ sheaf $\Ff$ on a smooth surface $S$ with canonical sheaf $K_S$ takes the form
\[
v = (v_1, v_2, v_3) = \left ( 0, \left [ \supp(\Ff) \right ], \chi(\Ff) + \frac{1}{2} \left [ K_S \right ] \cdot \left [ \supp(\Ff) \right ]  \right ),
\]
and the pairing is simply given by the intersection of the supports,
\[
\langle v , v' \rangle = v_2 \cdot v'_2 = \left [ \supp(\Ff) \right ] \cdot \left [ \supp(\Ff') \right ].
\]
Hence, the square of these Mukai vectors are determined by the genus of the support
\[
v^2 = 2 \cdot g \left ( \supp(\Ff) \right ) - 2.
\]

For this class of Mukai vectors, the associated moduli spaces can be equipped with a fibration,
\begin{equation} \label{eq LePoitier}
\map{\M_{S}^\H(v)}{\Hilb_S([nC])}{\Ff}{\supp(\Ff),}{}
\end{equation}
to the Hilbert scheme classifying dimension $1$ subschemes of $S$ with first Chern class equal to the second component of the Mukai vector. The fibre corresponding to a curve $A \in \Hilb_S([nC])$ amounts to the Simpson compactified Jacobian $\overline{\Jac}^{\, \H}_{A}(a)$, classifying $\H|_A$-semistable pure dimension $1$ sheaves on $A$ of rank $1$ and degree $a$.

\begin{remark} \label{rm Beauville-Mukai}
For a smooth projective K3 surface $X$ one has that $H^1(X,\Oo_X)$ vanishes so $\Pic(X)$ is discrete and embeds into $H^2(X,\ZZ)$. Then, the second component of the Mukai vector fixes the determinant and $\Hilb_X([nC])$ is the linear system $|nC|$. After the work of Beauville \cite{beauville_fibr} the fibres of \eqref{eq LePoitier} are Lagrangian with respect to the Mukai form $\Omega_\M$. This gives rise to an algebraically completely integrable system 
\begin{equation} \label{eq Mukai system}
\M_X^\H(v) \longrightarrow |nC|
\end{equation}
called the {\it Beauville--Mukai system}. 
\end{remark}

\section{O'Grady's exceptional examples}
\label{sc OGrady}

It is a classical statement that K3 surfaces and abelian surfaces are the only smooth projective surfaces that can be endowed with a holomorphic symplectic structure. Abelian surfaces are not simply connected but K3 are, so the later constitute the only example of compact irreducible holomorphic symplectic manifolds in dimension $2$. Beauville \cite{beauville_IHSM} produced two series of irreducible holomorphic symplectic manifold in any dimension, namely Hilbert schemes of points on a K3 surface, $\Hilb^n(X)$, and generalized Kummer varieties of an abelian surface, $\Kum^n(A)$. As we have seen in Section \ref{sc sheaves on surfaces}, Mukai \cite{mukai1} showed that the moduli space of torsion-free sheaves on a symplectic surface ({\it i.e.} K3 or abelian) can be equipped with a holomorphic symplectic form, so, whenever these moduli spaces are smooth they become examples of holomorphic symplectic manifolds. Nevertheless these spaces do not provide new examples of irreducible holomorphic symplectic manifolds since Yoshioka \cite{yoshioka_1, yoshioka_2} showed that, when $\M_X^\H(v)$ is smooth, it is deformation equivalent to a Hilbert scheme $\Hilb^n(X)$, and similar results hold in the case of $\M_A^\H(v)$ and generalized Kummer varieties. Hence, $\Hilb^n(X)$ and $\Kum^n(A)$ were the only known examples of irreducible holomorphic symplectic manifolds until the appearance of O'Grady's exceptional examples in dimension ten \cite{OGrady_1} and six \cite{OGrady_2}, whose constructions rely on the existence of symplectic desingularization of singular moduli spaces.

Consider a singular variety $M$ equipped with a holomorphic symplectic form $\Omega$ defined over the smooth locus $M^\smooth$. After Beauville \cite{beauville_sing} a symplectic desingularization of $M$ is a resolution of singularities $\widetilde{M} \to M$ equipped with a holomorphic symplectic form $\widetilde{\Omega}$ that extends the pull-back of $\Omega$ over the preimage of the smooth locus. As proved in \cite{beauville_sing}, when a symplectic desingularization exists, it is unique up to isomorphism.

For a particular choice of a primitive Mukai vector $v_0$ on a K3 surface $X$ that satisfies $v_0^2 = 2$, O'Grady considered the moduli space $\M_X^\H(2v_0)$, which is a singular variety of dimension $(2v_0)^2 + 2 = 10$. Using Kirwan desingularization procedure \cite{kirwan}, O'Grady \cite{OGrady_1} constructed a symplectic desingularization $\widetilde{\M}_X^\H(2v_0)$ and showed that it is an irreducible holomorphic symplectic variety.

Lehn and Sorger \cite{lehn&sorger} generalized O'Grady's symplectic desingularization of $\M_S^\H(2v_0)$ to any choice of primitive vector $v_0$ with $v_0^2= 2$ on a symplectic surface $S$. Furthermore, they showed that the symplectic resolution $\widetilde{\M}_S^\H(2v_0)$ of $\M_S^\H(2v_0)$ can be achieved by a single blow-up at the reduced subscheme of the singular locus which coincides with the locus of strictly polystable sheaves $\M_S^\H(2v_0)^\sps$.

It is a natural question to ask if one can obtain more holomorphic symplectic varieties from symplectic desingularization of moduli spaces of the form $\M_S^\H(nv_0)$. This was answer negatively by Kaledin, Lehn and Sorger \cite{kaledin&lehn&sorger} and Kiem \cite{kiem}, who proved non-existence of symplectic resolutions of $\M_S^\H(nv_0)$ for choices of $n$ and $v_0$ different than $n = 2$ and $v_0^2 = 2$. The work of Kaledin, Lehn and Sorger \cite{kaledin&lehn&sorger} is based in a local study of the singularities base on deformation theory, while the work of Kiem \cite{kiem} relies on the computation of the stringy $E$-functions of the moduli spaces.

We summarize all this in the following theorem.

\begin{theorem}[\cite{OGrady_1, OGrady_2, kiem, kaledin&lehn&sorger, lehn&sorger, perego&rapagnetta}] \label{tm resolution for surfaces}
Suppose $v = n v_0$ with $v_0^2 \geq 2$ and $n \geq 2$. Let $\H$ be a $v$-generic ample divisor on the symplectic surface $S$. Then $\M_S^\H(v)$ is a variety with symplectic singularities. 

Furthermore, if and only if $n = 2$ and $v_0^2 = 2$ there exists a symplectic resolution of $\M_S^\H(v)$ which can be obtained by a single blow-up at the reduced subscheme of the singular locus, which amounts to $\M_S^\H(v)^{\sps}$, the locus of strictly polystable sheaves. 
\end{theorem}

The symplectic resolutions appearing above allowed O'Grady to provide in \cite{OGrady_1} the first example of irreducible holomorphic symplectic manifold which is not deformation equivalent to $\Hilb^n(X)$ or $\Kum^n(A)$. Given a K3 surface $X$ and a particular choice of the primitive Mukai vector $v_0$ satisfying $v_0^2= 2$, O'Grady showed in \cite{OGrady_1} that the symplectic resolution $\widetilde{\M}^\H_X(2v_0)$ of the ten dimensional moduli space $\M^\H_X(2v_0)$ has second Betti number strictly bigger than $23$ and therefore is not deformation equivalent to the previously known irreducible holomorphic symplectic manifolds $\Hilb^5(X)$ and $\Kum^n(A)$, whose Betti numbers are smaller or equal to $23$. Rapagnetta \cite{rapagnetta} showed that the second Betti number of $\widetilde{\M}_X^\H(2v_0)$ is indeed $24$. Perego and Rapagnetta \cite{perego&rapagnetta} proved that the symplectic resolutions $\widetilde{\M}_X^\H(2v_0)$, obtained by Kaledin--Lehn--Sorger \cite{kaledin&lehn&sorger} for any choice of primitive Mukai vector $v_0$ with $v_0^2=2$, are all deformation equivalent.

For future notation convenience, we pick $X$ to be a K3 surface containing a smooth projective curve $C \subset X$ of genus $2$. Note that $C$ is an ample divisor of $X$ with self-intersection $C^2 = 2$, so the primitive Mukai vector $(0,C,\chi)$ satisfies $(0,C,-1)^2 = C^2 = 2$. We denote $\M_{\ten} := \M_X^{C}(0,2C,-2)$ and we refer to its symplectic resolution $\widetilde{\M}_{\ten}$ as the {\it O'Grady's ten dimensional space} (although, our $\widetilde{\M}_\ten$ is only deformation equivalent to the variety studied in \cite{OGrady_1}). Let us summarize all the above in the following theorem.

\begin{theorem}[\cite{OGrady_1, rapagnetta, lehn&sorger, perego&rapagnetta}] \label{tm ten-dimensional}
O'Grady's ten dimensional space $\widetilde{\M}_{\ten}$ is an irreducible holomorphic symplectic manifold with $b_2(\widetilde{\M}_{\ten}) = 24$, and therefore it is not deformation equivalent to the Hilbert scheme of lenth $5$ on a K3 surface, nor to the ten dimensional Kummer variety.

For any primitive Mukai vector satisfying $v_0^2 = 2$, there exists a symplectic resolution of $\M^\H_S(2v_0)$ and it is deformation equivalent to O'Grady's ten dimensional space $\widetilde{\M}_{\ten}$.
\end{theorem}

Consider now an abelian surface $A$. Let us denote its dual variety by $\widehat{A} := \M_A^\H(1,0,0)$ (for any choice of polarization $\H$), which is a fine moduli space for the classification of topologically trivial line bundles on $A$. Therefore, there exists a universal bundle $\P \to A \times \widehat{A}$ named the {\it Poincar\'e bundle}. Let us define the projections of the first factor $\f : A \times \widehat{A} \to A$ and to the second factor $\g : A \times \widehat{A} \to \widehat{A}$. Mukai considered in \cite{mukai_fourier} the corresponding integral functor, 
\begin{equation} \label{eq FM in A}
\morph{D^b(A)}{D^b(\widehat{A})}{\Ff^\bullet}{R \g_{!}(\P \otimes \f^*\Ff^\bullet),}{}{\Phi^{\P}}
\end{equation}
showing that it gives a derived equivalence. Fixing $\Ff_0 \in \M_S^\H(v)$, on gets the associated {\it Albanese morphism}
\begin{equation} \label{eq det and dual det for abelian surfaces}
\morph{\M_A^\H(v)}{A \times \widehat{A}}{\Ff}{( \det_{\widehat{A}}(\Phi^\P(\Ff)) \otimes \det_{\widehat{A}}(\Phi^\P(\Ff_0))^{-1} \, , \, \det_A(\Ff) \otimes \det_A(\Ff_0)^{-1} ).}{}{\a_{\Ff_0}}
\end{equation}
The pull-back through $\a_{\Ff_0}$ of the symplectic form $\Omega_{A \times \widehat{A}}$, provides a non-trivial holomorphic symplectic form on $\M_A^\H(v)$. This implies that holomorphic symplectic forms on this space are not unique up to scaling, so we can not obtain an irreducible holomorphic symplectic manifold from it. In view of this, one considers the subvariety given by the kernel of the Albanese maorphism,
\begin{equation} \label{eq definition of N}
\N_A^\H(v) := \a_{\Ff_0}^{-1} \left ( \id_A, \Oo_A \right ),
\end{equation}
which was studied in detail by Yoshioka \cite{yoshioka_1, yoshioka_2}. We denote by $\N_A^\H(v)^{\st}$ the locus of stable sheaves.

Recall the fibration \eqref{eq LePoitier} obtained by considering the fitting support of the sheaves. 
After the work of Beauville \cite{beauville_fibr} the fibres of this map are Lagrangian. One can check that his picture restricts naturally to the reduced moduli space $\N_A^\H(0,nC, \chi)$ over abelian surfaces,
\begin{equation} \label{eq reduced Mukai system}
\map{\N_A^\H(0,nC,\chi)}{|nC|}{\Ff}{\supp(\Ff),}{}
\end{equation}
giving rise to an algebraic completely integrable system that we call the {\it reduced Beauville-Mukai system}.

It can be shown that the Mukai form restricts to a symplectic form on $\N_A^\H(v)$ which is uniqe up to scaling. Furthermore, a symplectic resolution of $\M_A^\H(v)$ gives rise to a symplectic resolution of $\N_A^\H(v)$ in the cases studied in Theorem \ref{tm resolution for surfaces}, and that such a resolution only exists when the corresponding resolution for $\M_A^\H(v)$ exists as well. We then have an analogous to Theorem \ref{tm resolution for surfaces}.

\begin{theorem}[\cite{OGrady_2, kaledin&lehn&sorger, lehn&sorger}]\label{tm resolution for abelian surfaces}
Suppose $v = n v_0$ with $v_0$ primitive, $v_0^2 \geq 2$ and $n \geq 2$. Let $\H$ be a $v$-generic ample divisor on the abelian surface $A$. Then $\N_A^\H(v)$ is a variety with symplectic singularities. 

Furthermore, if and only if $n = 2$ and $v_0^2 = 2$ there exists a symplectic resolution of $\N_A^\H(v)$ which can be which can be obtained by a single blow-up at the reduced subscheme of the singular locus, which amounts to $\N_A^\H(v)^{\sps}$, the locus of strictly polystable sheaves.
\end{theorem}

O'Grady \cite{OGrady_2} provided another example of irreducible holomorphic symplectic manifolds by desingularizing $\N_A^\H(2,0,2\chi)$, in this case of dimension $6$. He also showed that its second Betti number is $8$, different from $\Hilb^3(X)$ and the six dimensional Kummer varieties, and therefore not deformation equivalent to them. As before, the work of Lehn and Sorger \cite{lehn&sorger} provides a symplectic resolution of $\M^\H_A(2v_0)$ by blowing-up the singular locus. Also, the results of Perego and Rapagnetta on deformation equivalence extend to this case. Again, for future notation convenience, we pick a smooth projective curve $C$ of genus $g=2$ and we take $A$ to be its Jacobian $A = \Jac(C)$. 
The genus formula fixes the self-intersection of our genus $2$ curve as $C^2 = 2$. Then, the primitive Mukai vector $(0,C,-1)$ squares to $2$. We denote $\M_{\six} := \M_A^{\H}(0,2C,-2)$ and $\N_{\six} := \N_A^{\H}(0,2C,-2)$, where we have chosen $\Ff_0$ to be a sheaf supported on the double curve $2C$. We say that the symplectic resolution $\widetilde{\N}_\six$ is the {\it O'Grady's six dimensional space} (even if $\widetilde{\N}_\six$ is only deformation equivalent to the one described in \cite{OGrady_2}). We summarize the previous discussion in the following theorem.

\begin{theorem}[\cite{OGrady_2, lehn&sorger, perego&rapagnetta}] \label{tm six-dimensional}
O'Grady's six dimensional space $\widetilde{\N}_{\six}$ is an irreducible holomorphic symplectic manifold with $b_2(\widetilde{\N}_{\six}) = 8$, and therefore it is not deformation equivalent to the Hilbert scheme of length $3$ on a K3 surface, nor to the six dimensional Kummer variety.

For any primitive Mukai vector satisfying $v_0^2 = 2$, there exists a symplectic resolution of $\N^\H_A(2v_0)$ and it is deformation equivalent to O'Grady's six dimensional space $\widetilde{\N}_{\six}$.
\end{theorem}

\section{The canonical ruled surface}
\label{sc canonical ruled surface}

Fixing a smooth projective curve $C$ of genus $g$ we denote its canonical bundle by $K$. We shall consider its total space $\Tot(K)$ and its projective compactification, the {\it canonical ruled surface}, namely the ruled surface with topological invariant $2g -2$, 
\[
\KK := \PP(\Oo_C \oplus K), 
\]
The canonical ruled surface plays an important role in connection with Higgs bundles, so we present in this section some of its main properties.

The canonical ruled surface is naturally equipped with the projection $p : \KK \to C$. One can easily construct two sections of $p$, namely $\sigma_0 \cong C$ that denotes the $0$-section of $K$, and the infinity section $\sigma_\infty \cong C$, given by the $0$-section of $\Oo_C$. Obviously, one has
\[
\Tot(K) = \KK - \{ \sigma_\infty \}.
\]

It is classical result, based on the properties of ruled surfaces (see \cite[Sect. 2, Chap. V]{hartshorne} for instance), that $H^2(\KK,\ZZ)$ is generated  by the classes of $\sigma_0$ and the fibres $F = p^{-1}(x_0)$, for some $x_0 \in C$, with 
\[
[\sigma_\infty]^2 = 2g-2
\]
and 
\[
[F]^2 = 0.
\]
Furthermore, the class of the infinity section is
\[
[\sigma_\infty] = [\sigma_0] - (2g-2)[F], 
\]
so it has negative self intersection
\[
[\sigma_\infty]^2 = -(2g-2),
\]
and,  
\[
[\sigma_0] \cdot [\sigma_\infty] = 0.
\]

One can now derive the following technical result that will be used in Section \ref{sc Higgs moduli spaces}.

\begin{lemma}\label{lm det Sigma = nC}
Let $\Sigma$ be a projective curve inside $\Tot(K) \subset \KK$, then
\[
\Oo_\KK(\Sigma) \cong \Oo_\KK(n\sigma_0)
\]
for some $n \geq 1$. In particular,
\begin{equation} \label{eq identification of support bases}
\left . \Hilb([n\sigma_0]) \right|_{\supp \cap \sigma_\infty  = \emptyset} \cong |n\sigma_0|_{\supp \cap \sigma_\infty  = \emptyset}.
\end{equation}
\end{lemma}

\begin{proof}
First we compute the class $[\Sigma]$ in $H^2(\KK, \ZZ)$. Since the later is generated by $[\sigma_0]$ and $[F]$, one has
\[
[\Sigma] = \alpha [\sigma_0] + \beta [F]
\]
for some $n, m \in \ZZ$. By hypothesis, $\Sigma$ is contained in $\Tot(K)$, what implies that
\[
[\Sigma] \cdot [\sigma_\infty] = 0.
\]
Hence, 
\[
0 = \left ( n [\sigma_0] + m [F]\right ) \cdot \left ( [\sigma_0] - (2g-2) [F]\right ) =  \beta,
\]
so
\[
[\Sigma] = [n\sigma_0].
\]

Set $L := \Oo_\KK(\Sigma) \otimes \Oo_\KK(nC)^{-1}$. We shall next see that $L$ is indeed the trivial bundle. Note that one has a $\CC^*$ action on $\KK$ given by scalling the fibres by the multiplication of $\tau \in \CC^*$. We see that under this action $\Sigma$ tends, when $\tau \to 0$, to a curve supported on $\sigma_0$, which must be the multiple curve $n\sigma_0$ as $[\Sigma] = [nC]$. Hence, one can construct a family $\Xx$ of curves inside $\KK$ parametrized by $\CC$, whose fibre over $\tau \in \CC^*$ is $\tau(\Sigma)$ and its fibre over $t=0$ is the multiple curve $nC$. Hence, 
\[
\Ll = \Oo_{\KK \times \CC}\left (\Xx - (n\CC \times \CC)\right)
\]
provides a family of line bundles, satisfying
\[
\Ll|_{\tau \neq 0} \cong \Oo_\KK(\tau(\Sigma)) \otimes \Oo_\KK(nC)^{-1} \cong (\tau^{-1})^*\Oo_\KK(\Sigma) \otimes \Oo_\KK(nC)^{-1} \cong \Oo_\KK(\Sigma) \otimes \Oo_\KK(nC)^{-1} \cong p^*L
\]
and
\[
\Ll|_{\tau = 0} \cong \Oo_\KK(nC) \otimes \Oo_\KK(nC)^{-1} \cong \Oo_\KK.
\]  
It then follows that $L \cong \Oo_\KK$, what concludes the proof of the first statement. The second statement follows inmediately from the first.  
\end{proof}

We will now study the open subset of the linear system $|n\sigma_0|_{\supp \cap \sigma_\infty  = \emptyset}$. This open subset of the linear system $|n\sigma_0|$ contains curves that can be obtained by linear deformations of the multiple curve $n\sigma_0$. Linear deformations of a curve are classified by the sections of the normal bundle restricted to the curve. Hence, one has the identification
\begin{equation} \label{eq identification of linearsystem with deformations}
|n\sigma_0|_{\supp \cap \sigma_\infty = \emptyset} = H^0(n\sigma_0, \Oo_{\KK}(n\sigma_0)|_{n\sigma_0}).
\end{equation}

Let $r$ denote the obvious projection of the multiple curve onto its reduced support and observe that $r_*\Oo_{\KK}(n\sigma_0)|_{n\sigma_0}$ amounts to $\Oo_{\KK}(n\sigma_0)|_{\sigma_0} \otimes r_*\Oo_{n\sigma_0}$ by the projection formula. The structural sheaf $\Oo_{n\sigma_0} \cong \Oo_\Sigma/\Ii_C^n$ decomposes, as an $\Oo_C$-module, into $\Oo_C \oplus (\Ii_C/\Ii_C^2) \oplus \dots \oplus (\Ii_C^{n-1}/\Ii_C^{n})$. The intersection of $\sigma_0$ with the canonical divisor is zero, so $\Ii_C/\Ii_C^2 \cong \Oo_{\KK}(-\sigma_0)|_{\sigma_0} \cong K^{-1}$. Therefore,
\[
H^0(n\sigma_0, \Oo_{\KK}(n\sigma_0)|_{n\sigma_0}) \cong \bigoplus_{i = 1}^n H^0(C,K^{\otimes i}),
\]
so the open subset of the linear system $|n\sigma_0|_{\supp \cap \sigma_\infty  = \emptyset}$ can be identified with
\begin{equation} \label{eq identification of Hitchin bases}
|n\sigma_0|_{\supp \cap \sigma_\infty  = \emptyset} \cong \bigoplus_{i = 1}^n H^0(C,K^{\otimes i}).
\end{equation}

We finish this section with the study of the Poisson structure of $\KK$. It is well known (see \cite[Sect. 2, Chap. V]{hartshorne}) that its canonical divisor is
\[
K_\KK = \Oo_\KK(-2 \sigma_\infty).
\]
It follows that $\KK$ is a Poisson surface, with $\Theta_\KK$ being a section of $\bigwedge^2 \Tt_\KK \cong K_\KK^{-1} \cong \Oo_\KK(2 \sigma_\infty)$ vanishing with order $2$ in $\sigma_\infty$. Furthermore, $\Tot(K) = \KK \setminus \{ \sigma_\infty \}$ is equipped with a holomorphic symplectic form $\Omega_0$ as it is a cotangent bundle, and the later coincides with the one induced by the Poisson structure.

\section{Moduli spaces of Higgs bundles and their symplectic resolutions}
\label{sc Higgs moduli spaces}

Let $C$ be a smooth projective curve. A {\it Higgs bundle} is a pair $(E, \varphi)$, where $E$ is a holomorphic vector bundle on $C$, and $\varphi \in H^0(C, \End(E) \otimes K)$ is a holomorphic section of the endomorphisms bundle, twisted by the canonical bundle $K$. It is possible to construct \cite{hitchin-self, simpson1, simpson2, nitsure} the moduli space of rank $n$ and degree $d$ semistable (resp. stable) Higgs bundles on $C$ which we denote by $\Mm_C(n,d)$ (resp. $\Mm_C(n,d)^{\st}$). This is a quasi-projective variety of dimension $2n^2(g-1) + 2$. It is also connected, normal and irreducible after Simpson \cite{simpson2}. The moduli space of stable Higgs bundles $\Mm_X(n,d)^{\st}$ is a smooth open subset of $\Mm_X(n,d)$.

Taking a basis $(q_1, \dots, q_n)$ of the $\GL(n,\CC)$-invariant polynomials with $\deg(q_i) = i$, one can construct the {\it Hitchin fibration},
\begin{equation} \label{eq Hitchin system}
\map{\Mm_C(n,d)}{B_n := \bigoplus_{i = 1}^n H^0(C,K^{\otimes i})}{(E,\varphi)}{\left ( q_1(\varphi), \dots, q_n(\varphi) \right ).}{}
\end{equation}
which is a surjective proper morphism \cite{hitchin_duke} endowing $\Mm_C(n,d)$ with the structure of an algebraically completely integrable system, called the {\it $\GL_n$-Hitchin system}. In particular, the generic fibers of \eqref{eq Hitchin system} as abelian varieties that admit a beautiful description \cite{hitchin_duke, BNR, simpson2, schaub} in terms of (compactified) Jacobians of certain curves, called {\it spectral curves} in the literature, covering our base curve $C$. To describe these, consider the total space $\Tot(K)$ of the canonical bundle and the surjection $p: \Tot(K) \to C$ obtained by the structure morphism of the line bundle $K$ over $C$. Observe that the pullback bundle $p^*K \to \Tot(K)$ admits a tautological section $\lambda$ sending every point $k \in \Tot(K)$ to the point $k$ on the fibre of $p^*K$ over $p(k)$. For every $b = (b_1, \dots, b_n) \in B_n$, with $b_i \in H^0(C,K^i)$, we construct the {\it spectral curve} $C_b \subset \Tot(K)$ by considering the vanishing locus of the section of $\pi^*K^n$
\[
\lambda^n + \pi^*b_1 \lambda^{n-1} + \dots + \pi^*b_{n-1} \lambda + \pi^*b_n.
\]
The restriction of $p$ to $C_b$ is a ramified degree $n$ cover over $C$ that we denote by
\[
p_b : C_b \longrightarrow C.
\]
The genus formula and the triviality of the canonical bundle of the quasiprojective surface $\Tot(K)$, allow us to compute the genus $g(C_b) = n^2(g-1) + 1$. 

Let us now consider a torsion-free rank one sheaf $\Ff$ over $C_b$ of Euler characteristic $d - n(g-1)$ and whose restriction to every irreducible component of $C_b$ has rank $1$. One has that the push-forward $E_\Ff := \pi_{b,*}\Ff$ is a vector bundle on $C$ of rank $n$ and degree $d$. Furthermore, pushing-forward the morphism associated to tensorization by the tautological section 
\[
\lambda : \Ff \to \Ff \otimes \pi^*K,
\]
one obtains, thanks to the projection formula, a Higgs field for our vector bundle $E_\Ff = p_{b,*}\Ff$, 
\begin{equation} \label{eq Higgs field}
\varphi_\Ff:= p_{b,*}\lambda : E_\Ff \longrightarrow E_\Ff \otimes K.
\end{equation}
One can naturally check that the image of $(E_\Ff, \varphi_\Ff)$ under \eqref{eq Hitchin system} is $b$. The previous construction is indeed a $1$-$1$ correspondence, called the {\it spectral correspondence}, between rank $1$ torsion free sheaves on $C_b$ and Higgs bundles whose image under the Hitchin morphism is $b \in B_n$ \cite{BNR}.

Furthermore, the (semi)stability of the rank $1$ torsion-free sheaves $\Ff$ on  $C_b$ is equivalent to the (semi)stability of the corresponding Higgs bundle $(E_\Ff, \varphi_\Ff)$. Therefore, the fibres of \eqref{eq Hitchin system} are identified with the compactified Jacobians of the corresponding spectral curve, $\overline{\Jac}^{\, \delta + d}(C_b)$, classifying torsion-free sheaves on $C_b$ with rank $1$ on each irreducible component of $C_b$.

Note that all the spectral curves $C_b$ belong to the linear system $|n\sigma_0|$ in $\Tot(K)$.
Over the ruled surface $\KK$ we consider the moduli space of sheaves $\M_\KK^\H(0,n\sigma_0, d-n(g-1))$ and inside it, take the open locus $\left . \M^{\H_0}_{\KK}(0,n\sigma_0, d-n(g-1)) \right|_{\supp \cap \sigma_\infty = \emptyset}$ of those sheaves that do not intersect the infinity section $\sigma_\infty$. After Lemma \ref{lm det Sigma = nC} and \eqref{eq identification of Hitchin bases}, the support morphism \eqref{eq LePoitier} restricts to 
\[
\left . \M^{\H_0}_{\KK}(0,n\sigma_0, d-n(g-1)) \right|_{\supp \cap \sigma_\infty = \emptyset} \longrightarrow B_n \cong |n\sigma_0|_{\supp \cap \sigma_\infty = \emptyset}.
\]

In this context, the spectral correspondence leads to a second construction of the Higgs moduli space as the closed and dense subset of $\left . \M^{\H_0}_{\KK}(0,n\sigma_0, d-n(g-1)) \right|_{\supp \cap \sigma_\infty = \emptyset}$ given by those sheaves whose restriction to its associated spectral curve has rank $1$ on each irreducible component,
\begin{equation} \label{eq compactification of Higgs space}
\Mm_C(n,d) = \left \lbrace \begin{matrix} \Ff \in \M^{\H_0}_{\KK}(0,n\sigma_0, d-n(g-1)) \textnormal{ s. t. } \\ \supp(\Ff) \subset \Tot(K) = \KK \setminus \sigma_\infty \textnormal{ and } \\ \rk(\Ff|_{Y_i}) = 1 \textnormal{ for every irreducible} \\ \textnormal{component } Y_i \textnormal{ of } \supp(\Ff) \end{matrix} \right \rbrace.
\end{equation}
Following \eqref{eq identification of Hitchin bases}, under this identification, the Hitchin fibration can be interpreted as the support fibration \eqref{eq LePoitier},
\begin{equation} \label{eq Hitchin system 2}
\Mm_C(n,d) \longrightarrow |n\sigma_0|_{\supp \cap \sigma_\infty = \emptyset}.
\end{equation}

Fixing a point $x_0 \in C$, one can define the natural morphism
\[
\morph{\Mm_C(n,d)}{\Jac^0(C)\times H^0(C,K)}{(E,\varphi)}{\left( \det(E) \otimes \Oo_C(d x_0)^{-1}, \tr(\varphi) \right).}{}{a_{x_0}}
\]
The kernel of this map is the moduli space of $\SL_n(n,\cC)$--Higgs bundles,
\[
\Nn_C(n,d) := a_{x_0}^{-1}(\Oo_C, 0).
\]
The locus of stable points $\Nn_C(n,d)$ is a smooth open subset of $\Mm_C(n,d)$. As before, $\Nn_C(n,d) = \Nn_C(n,d)^{\st}$ is smooth if and only if $n$ and $d$ are coprime. The Hitchin morphism \eqref{eq Hitchin system} restricts to 
\begin{equation} \label{eq reduced Hitchin system}
\map{\Nn_C(n,d)}{V_n := \bigoplus_{i = 2}^n H^0(C,K^{\otimes i})}{(E,\varphi)}{\left ( q_2(\varphi), \dots, q_n(\varphi) \right ),}{}
\end{equation}
becoming an algebraically completely integrable system \cite{hitchin_duke} that we call the {\it $\SL_n$-Hitchin system}.

One can define a holomorphic symplectic form on $\Mm_C(n,d)$ out of the construction of this moduli space as an infinite hyperK\"ahler quotient. This holomorphic symplectic form coincides (up to scalling) with that obtained out of the Bottacin--Markman Poisson structure on the moduli space \cite{bottacin_2, markman}. Also, the holomorphic symplectic form is the extension of the standard holomorphic symplectic form defined on the cotangent of the moduli space of stable vector bundles, which is naturally included in $\Mm_C(n,d)$.

Using deformation theory \cite{hitchin-self, simpson2, nitsure}, one can compute the tangent space of the moduli spaces $\Mm_C(n,d)$ and $\Nn_C(n,d)$. Since stable implies simple also in this case, the stable loci $\Mm_C(n,d)^\st$ and $\Nn_C(n,d)^\st$ immediately smooth. A less straigh-forward analysis based on DGLA was used by Simpson \cite{simpson2} to show that the singular locus of $\Mm_C(n,d)$ and $\Nn_C(n,d)$ are indeed their strictly polystable loci $\Mm_C(n,d)^\sps$ and $\Nn_C(n,d)^\sps$.

The study of the existence of symplectic resolutions of the Higgs moduli spaces was first carried out by Kiem and Yoo \cite{kiem&yoo} for $\Nn_C(2,0)$ and different values of $g$, the genus of the curve. The method used here consists on the computation of the stringy E-function, which is a polynomial if and only if there exists a crepant resolution (recall that symplectic resolution implies crepant). Tirelli provided in \cite{tirelli} a complete description of the occurrence of symplectic resolutions of the Higgs moduli space for any rank and any genus. The results of \cite{tirelli} rely on Simpson's Isosingularity Theorem \cite{simpson2} and the work of Bellamy and Schedler \cite{bellamy&schedler} on symplectic resolutions of character varieties associated to Riemann surfaces, generalizing the approach of Kaledin--Lehn--Sorger \cite{lehn&sorger, kaledin&lehn&sorger} to the context of quiver and character varieties. 

\begin{theorem}[\cite{kiem&yoo, tirelli, bellamy&schedler}] \label{tm desingularization} 
Let $g$ denote the genus of the smooth projective curve $C$ and assume $g \geq 2$. Then, $\Mm_C(n,0)$ and $\Nn_C(n,0)$ are varieties with a symplectic singularities and  
\begin{enumerate}
\item when $(n,g) = (2,2)$, there exist symplectic resolutions $\widetilde{\Mm}_C(2,0)$ and $\widetilde{\Nn}_C(2,0)$ of, respectively, $\Mm_C(2,0)$ and $\Nn_C(2,0)$ obtained after blowing-up the respective singular loci which amounts to their strictly polystable loci $\Mm_C(2,0)^\sps$ and $\Nn_C(2,0)^\sps$;
\item when $(n,g)\neq (2,2)$, there does not exist a symplectic resolution of $\Mm_C(n,0)$ nor $\Nn_C(n,0)$.
\end{enumerate}
\end{theorem}


\section{The non-linear deformation of Donagi, Ein and Lazarsfeld}
\label{sc degeneration of ten dimensional}

In this section we review the non-linear deformation of Donagi, Ein and Lazarsfeld \cite{donagi&ein&lazarsfeld} for the case of K3 surfaces, later revisited by de Cataldo, Maulik and Shen in \cite{deCataldo&maulik&shen_2} in the case of abelian surfaces (see the Appendix of \cite{felissetti&mauri} as well). This degeneration was generalized by the author  to the case of arbitrary smooth surfaces in \cite{franco}, where, furthermore, it was described the behaviour of the holomorphic symlectic form.

Take a smooth symplectic surface $S$ containing a genus $g > 1$ smooth curve 
\[
C \subset S.
\]
Thanks to the triviality of the canonical bundle $K_S$, the adjunction formula implies that the normal bundle of $C$ in $S$ coincides with the canonical bundle of $C$, 
\[
N_{S/C} \cong K.
\]
Consider the projective compactification $\KK$. Recalling the properties of the canonical ruled surface from Section \ref{sc canonical ruled surface}, one has $[\sigma_0]^2 = 2g-2 = [C]^2$ and $[\sigma_0] \cdot [K_\KK] = 0 = [C] \cdot [K_S]$. Hence, the linear systems $|nC|$ and $|n\sigma_0|$ contain generic curves of equal genus,
\begin{equation} \label{eq equal genus}
g_{nC} = g_{n\sigma_0},
\end{equation}
and the topological invariants of the ideal sheaves describing these curves coincide.

Take the trivial family $S \times \PP^1$ and consider the blow-up at $C \times \{ 0 \}$,
\begin{equation} \label{eq definition of Ss}
\ol{\sS} := \Blow_{C \times \{ 0 \} }\left (S \times \PP^1 \right ),
\end{equation}
which is non-singular as we are blowing-up a non-singular subvariety. Composing the structural morphism of the blow-up with the projection $S \times \PP^1 \to \PP^1$, one gets the structural morphism 
\[
\pi : \ol{\sS} \to \PP^1. 
\]
This map is flat and its fibres outside $0 \in \PP^1$ are not affected by our construction, so for the generic fibre, at $t \neq 0$, we have
\[
\ol{\sS}_t \cong S.
\]
Meanwhile, the central fibre, at $t = 0$, is the union
\begin{equation} \label{eq definition of Ss_0}
\ol{\sS}_0 = \KK \cup S,
\end{equation}
where both irreducible components meet transversely at the image of infinity section $\sigma_\infty \subset \LL$ identified with the curve $C \subset S$. Note that $\ol{\sS}_0$ is a complete intersection subvariety. 
One also consider the open subset
\begin{equation} \label{eq definition of Ss'}
\sS := \ol{\sS} - ( S \times  \{ 0 \}),
\end{equation}
and observe that the restriction there of the structural morphism gives, again, a flat family of surfaces over $\PP^1$, but now the central fibre $\sS|_{t = 0} = \Tot(K)$ is not projective.

Denote by $p_S$ the composition of the structural morphism of the blow-up with the projection of $S \times \PP^1$ onto $S$. Being $\ol{\sS}$ projective, one may choose a relative polarization
\begin{equation} \label{eq definition of Hh}
\hH = p_S^* \H \otimes \Oo_{\ol{\sS}}(\KK). 
\end{equation}
For all $t \neq 0$ in $\PP^1$, one trivially has that $\hH|_{\ol{\sS}_t} = \H$ while $H_0 := \hH|_{\KK}$ corresponds to $p^*(\H|_C) \otimes \Oo_{\KK}(\sigma_{\infty})$.

Since $C \times \{ 0 \}$ is a Cartier divisor of $C \times \PP^1$, one has that 
\begin{equation} \label{eq definition of Cc}
\cC := \Blow_{C \times \{ 0 \}}\left ( C \times \PP^1 \right ) \cong C \times \PP^1
\end{equation}
embeds naturally into $\ol{\sS}$. One can consider the fibration $\cC \to \PP^1$, whose generic fibres, at $t\neq 0$, are naturally identified with $C$,
\[
\cC_t \cong C \subset S = \ol{\sS}_t.
\]
Furthermore, the restriction of $\cC$ to $t = 0$ coincides with the image of the zero section,
\[
\cC_0 \cong \sigma_0 \cong C, 
\]
inside $\LL \subset \ol{\sS}_0$. Note that the restriction to every point of $\PP^1$ gives $\cC_t \in H^2(\ol{\sS}_t, \ZZ)$. Using $\cC$, consider the Mukai vector
\begin{equation} \label{eq definition of v}
\vV = \left ( 0, [n\cC], d - n(g-1)   \right ),
\end{equation}
for some integer $d \in \ZZ$.

Following Simpson \cite[Theorem 1.21]{simpson1} (see \cite[Theorem 4.3.7]{huybrechts&lehn} for a specific treatment of the relative case), let us consider the moduli space $\M_{\ol{\sS}/\PP^1}^\hH(\vV)$
of relatively pure dimension $1$ $\hH$-semistable relative sheaves over the $\PP^1$ scheme $\ol{\sS}$ having Mukai vector $\vV$. By construction, $\M_{\ol{\sS}/\PP^1}^\hH(\vV)$ comes equipped with the structural flat morphism to $\PP^1$, which trivializes over $\PP^1 - \{ 0 \}$ having generic fibres equal to $\M_{S,\H}(\v)$, and whose central fibre over $t = 0$ is $\M^{\hH_0}_{\KK \cup S}(\v)$. 

Let us consider the relative Hilbert scheme $\Hilb_{\ol\sS/\PP^1}([n\cC])$ parametrizing ideal sheaves of curves in $\ol{\sS}$ with first Chern class $[n\cC]$ and let us denote by $\Hilb_{\sS/\PP^1}([n\cC])$ the open subset given by those curves contained in $\sS$. The determinant in cohomology is well defined and one can consider 
\begin{equation} \label{eq det_Ss}
\det_{\sS} : \Hilb_{\ol\sS/\PP^1}([n\cC]) \to \Jac_{\PP^1}(\ol{\sS}),
\end{equation}
that restricts, of course, to $\Hilb_{\sS/\PP^1}([n\cC])$. For $t \neq 0$, the fibre of \eqref{eq det_Ss} over $L \in \Jac(S)$ corresponds to the linear system $|L|$ in $S$. After Lemma \ref{lm det Sigma = nC}, one has that the whole slice of the Hilbert scheme is sent to a point,
\begin{equation} \label{eq the whole slice of Hilb is sent to a point}
\det_{\sS} \left ( \Hilb_{\sS/\PP^1}([n\cC])|_{t = 0} \right ) = \Oo_{\ol{\sS}_0}(n \sigma_0).
\end{equation}
Furthermore, making use of \eqref{eq identification of Hitchin bases}, one can identify the whole central fibre of the Hilbert scheme with the Hitchin base
\[
\Hilb_{\sS/\PP^1}([n\cC])|_{t = 0} \cong \bigoplus_{i = 1}^n H^0(C, K^{\otimes i}) = B_n. 
\]

By specifying the relative (fitting) support of the sheaves in $\M_{\ol{\sS}/\PP^1}^\hH(\vV)$, one has the following surjective morphism
\begin{equation} \label{eq relative Beauville fibration}
\morph{\M_{\ol{\sS}/\PP^1}^\hH(\vV)}{\Hilb_{\ol{\sS}/\PP^1}([n\cC])}{\Ff}{\supp(\Ff).}{}{\mathbf{h}}
\end{equation}
The morphism $\mathbf{h}$ commutes with the structural morphisms (to $\PP^1$). Let us consider the restriction to those sheaves supported in $\sS \subset \ol\sS$,
\[
\mathbf{h}^{-1}\left ( \Hilb_{\sS/\PP^1}([n\cC]) \right ) \longrightarrow \Hilb_{\sS/\PP^1}([n\cC]).
\]
Observe that $\mathbf{h}^{-1}\left ( \Hilb_{\sS/\PP^1}([n\cC]) \right ) \longrightarrow \PP^1$ has generic fibres $\M_{S}^\H(v)$ over $t \neq 0$ and central fibre $\M^{H_0}_{\KK}(v)|_{\supp \cap \sigma_\infty = \emptyset}$ at $t = 0$. The later is a consequence of the equivalence between the $H_0$-semistability of a sheaf in $\Tot(K) \subset \KK$ and its $\hH_0$-semistability as a sheaf in $\Tot(K) \subset \KK \cup S$. Having \eqref{eq compactification of Higgs space} in mind, we define the closed and dense subset of $\mathbf{h}^{-1}\left ( \Hilb_{\sS/\PP^1}([n\cC]) \right ) \subset \M_{\ol{\sS}/\PP^1}^\hH(\vV)$, given by those sheaves that are either supported in a generic fibre, or those supported in the central fibre having rank $1$ on each irreducible component of its support,
\begin{equation} \label{eq definition of Mm}
\mM_{\sS/\PP^1}^\hH(0, n\cC, d-n(g-1)) := \left \lbrace \begin{matrix} \Ff \in \M_{\ol{\sS}/\PP^1}^\hH(\vV) \textnormal{ s. t. } \supp(\Ff) \subset \sS  \textnormal{ and, whenever} \\  \supp(\Ff) \subset \Tot(K) \subset \ol\sS_0, \textnormal{ one has } \rk(\Ff|_{Y_i}) = 1 \\  \textnormal{for every irreducible component } Y_i \textnormal{ of } \supp(\Ff) \end{matrix} \right \rbrace.
\end{equation}

We can now see that the above construction provides the degeneration that we were looking for. This degeneration appeared first in \cite{donagi&ein&lazarsfeld} for K3 surfaces and in \cite{deCataldo&maulik&shen_1} for abelian surfaces. A generalization to arbitrary smooth surfaces is provided in \cite{franco}.

\begin{theorem}[\cite{donagi&ein&lazarsfeld}] \label{tm Mm_S}
Consider a smooth symplectic surface surface $S$ and a genus $g \geq 2$ smooth projective curve $C$ inside $S$. One can construct $\mM^\hH_{\sS}(0, n\cC, d-n(g-1))$ and $\Hilb_{\ol{\sS}/\PP^1}([n\cC])$ flat over $\PP^1$, and a morphism between them commuting with the structural morphisms, 
\begin{equation} \label{eq commuting diagram Mm_S}
\xymatrix{
\mM^\hH_{\sS/\PP^1}(0, n\cC, d - n(g-1)) \ar[rr] \ar[rd] & & \Hilb_{\sS/\PP^1}([n\cC]) \ar[ld]
\\
& \PP^1. &
}
\end{equation}
The horizontal arrow is a surjective fibration, trivial over $\PP^1 - \{ 0 \}$, and the generic fibre over $t \neq 0$ is the support morphism \eqref{eq LePoitier}, 
\[
\M_S^\H(0,nC, d - n(g-1)) \to \Hilb_{S}([n\cC]),
\]
while the central fibre at $t = 0$ gives the Hitchin morphism \eqref{eq Hitchin system},
\[
\Mm_C(n,d) \to B_n = \bigoplus_{i = 1}^n H^0(C, K^{\otimes i}).
\]
\end{theorem}

Following \cite{franco}, we now review the study of the behaviour of the symplectic structure along Donagi--Ein--Lazarsfeld degeneration, showing that it provides a deformation of symplectic varieties. We start by studying $K_{\ol{\sS}/\PP^1}$ and $K_{\ol{\sS}_0}$ in the next proposition. For the shake of completion, we reproduce its proof that first appeared in \cite{franco}.

\begin{proposition} \label{pr relative canonical bundle}
Let $S$ be a symplectic surface and pick $C \subset S$ smooth curve of genus $g \geq 2$. One has the following
\begin{enumerate}

\item \label{it description of omega_Ss_0} $K_{\ol{\sS}_0}$ is the line bundle given by $\Oo_\KK(-\sigma_\infty)$ and $\Oo_S(C)$ identified along $\sigma_\infty \cong C$ by a natural isomorphism between $K \cong \Oo_\KK(-\sigma_\infty)|_{\sigma_\infty}$ and $K \cong \Oo_S(C)|_C$;

\item \label{it sections of omega} $H^0(\ol{\sS}_0, K_{\ol{\sS}_0}) = \cC$ and every non-trivial section of $K_{\ol{\sS}_0}$ vanishes on $\KK$ and only on $\KK$;

\item \label{it sections of omega^-1} $H^0(\ol{\sS}_0, \omega^{-1}_{\ol{\sS}_0}) = \cC$ and every non-trivial section of $\omega^{-1}_{\ol{\sS}_0}$ vanishes on $S$ and only on $S$; 

\item \label{it pi_* omega} $\pi_*K_{\ol{\sS}/\PP^1} \cong \Oo_{\PP^1}(1)$;

\item \label{it pi_* omega^-1} $\pi_*\omega^{-1}_{\ol{\sS}/\PP^1} \cong \Oo_{\PP^1}(-1)$.

\end{enumerate}
\end{proposition}

\begin{proof}
We first observe that $\ol{\sS}$ is constructed by blowing-up a smooth subvariety of $S \times \PP^1$, which is also smooth. Hence $\ol{\sS}$ is smooth and has canonical bundle $K_{\ol{\sS}} = r^*K_{S \times \PP^1}(\KK)$, where we denote by $r$ the structural morphism of the blow-up and we recall that $\KK$ is the exceptional divisor. One can deduce that the singular central fibre $\ol{\sS}_0 = \KK \cup S$, where both components meet transversally, is a complete intersection variety, so it is Gorenstein and its dualizing sheaf $K_{\ol{\sS}_0}$ is a line bundle. This implies that the relative dualizing sheaf $K_{\ol{\sS}/\PP^1}$ is a line bundle too, indeed it is $\Oo_{\ol{\sS}}(\KK)$, as $K_S$ is trivial. 

Since $\ol{\sS}_0$ is a complete intersection divisor in $\ol{\sS}$, one has that $K_{\ol{\sS}_0} = K_{\ol{\sS}/\PP^1}(\KK + S) |_{\ol{\sS}_0}$, where we abuse of notation by denoting the total transform of $S \times \{ 0 \}$ simply by $S$. Then, 
\[
K_{\ol{\sS}_0}|_\KK \cong \Oo_{\ol{\sS}}(2\KK)|_\KK \otimes \Oo_{\ol{\sS}}(S)|_\KK \cong \Oo_\KK(-2\sigma_\infty) \otimes \Oo_\KK(\sigma_\infty) \cong \Oo_\KK(-\sigma_\infty),
\] 
where we recall that $\Oo_{\ol{\sS}}(-\KK)|_\KK$ is the linearization $\Oo_\KK(1) \cong \Oo_\KK(\sigma_\infty)$. Similarly, the restriction of $K_{\ol{\sS}_0}$ to $S$ gives
\[
K_{\ol{\sS}_0}|_S \cong \Oo_{\ol{\sS}}(2\KK)|_S \otimes \Oo_{\ol{\sS}}(S)|_S \cong \Oo_{S}(2C) \otimes \Oo_{S}(-C) \cong \Oo_S(C),
\]
where we recall that $K_S \cong \Oo_{\ol{\sS}}(\KK + S)|_S \cong \Oo_{\ol{\sS}}(\KK)|_S \otimes \Oo_{\ol{\sS}}(S)|_S$ is trivial, hence $\Oo_{\ol{\sS}}(S)|_S \cong (\Oo_{\ol{\sS}}(\KK)|_S)^{-1}$. Recalling that $K_\KK = \Oo_\KK(-2\sigma_\infty)$ and $K_S \cong \Oo_S$, applying adjunction, one gets that $\Oo_\KK(-\sigma_\infty)|_{\sigma_\infty}$ and $\Oo_S(C)|_C$ both being isomorphic to $K$, the canonical bundle of $\sigma_\infty \cong C$, so the restriction of the line bundle $K_{\ol{\sS}_0}$ provides a natural identification. This finishes the proof of \eqref{it description of omega_Ss_0}. 

Since $\Oo_\KK(-\sigma_\infty)$ has no non-zero sections, every non-zero section of $K_{\ol{\sS}_0}$ vanishes completely on $\KK$, hence it is given by a section of $\Oo_S(C)$ vanishing identically at $C$. The set of those sections determines a $1$-dimensional subspace of $H^0(S,\Oo_S(C))$ and \eqref{it sections of omega} follows. 

We proof \eqref{it sections of omega^-1} by observing that $h^0(\KK, \Oo_\KK(\sigma_\infty)) = 1$ and that $h^0(S, \Oo_S(-C)) = 0$. It follows that every non-zero section of $K_{\ol{\sS}_0}^{-1}$ vanishes on $S$, so it is given by a section of $\Oo_\KK(\sigma_\infty)$ which vanishes identically at $\sigma_\infty$.

Recall that $\ol{\sS}|_{\PP^1 - \{ 0 \}}$ is the trivial fibration $S \times (\PP^1 - \{ 0 \})$, so  the restriction there of $K_{\ol{\sS}/\PP^1}$ is the trivial line bundle. After \eqref{it sections of omega} and \eqref{it sections of omega^-1} and upper semicontinuity of cohomology, one has that $\pi_*K_{\ol{\sS}/\PP^1}$ and $\pi_*\omega^{-1}_{\ol{\sS}/\PP^1}$ are line bundles over $\PP^1$, inverse to each other. Note that $K_{\ol{\sS}/\PP^1} \cong \Oo_{\ol{\sS}}(\KK)$ comes naturally equipped with a section. Furthermore, $h^0(\Oo_{\ol{\sS}}(\KK)) = 1$ by the properties of blow-up. It follows that $\pi_*K_{\ol{\sS}/\PP^1}$ has a single non-zero section (up to scaling), so \eqref{it pi_* omega} follows. 

Finally, \eqref{it pi_* omega^-1} follows from \eqref{it pi_* omega}.  
\end{proof}

In view of \eqref{it sections of omega} and \eqref{it sections of omega^-1} of Proposition \ref{pr relative canonical bundle} we shall first construct a Poisson structure on $\M_{\ol{\sS}/\AA^1}^\hH(\vV)$ and, then, derive the symplectic structure on $\mM_{\sS/\AA^1}^\hH(0, n\cC, d-n(g-1))$. This is the context of the following theorem that first appeared in \cite{franco}, and whose proof we reproduce for the shake of completion.

\begin{theorem} \label{tm relative symplectic form}
Consider a smooth symplectic surface $S$ and a genus $g \geq 2$ smooth projective curve $C$ inside $S$. Then, there exists a relative Poisson structure $\Theta$ on $\M_{\ol{\sS}/\AA^1}^\hH(\vV) \to \AA^1$ which defines a relative symplectic form $\varOmega$ on $\mM_{\sS/\AA^1}^\hH(0, n\cC, d-n(g-1)) \to \AA^1$ coinciding (up to scaling) with the Mukai form $\Omega$ over the generic fibre $\M_{S}^\H(0,nC,d-n(g-1))$ over $t \neq 0$, and, on the central fibre $\Mm_C(n,d)$ at $t = 0$, with $\Omega_0$ obtained by extending the canonical symplectic form on the cotangent of the moduli space of stable vector bundles.
\end{theorem}

\begin{proof}
Starting from \eqref{it pi_* omega^-1} of Proposition \ref{pr relative canonical bundle}, we pick $\AA^1 \subset \PP^1$ a trivialization of $\pi_*K_{\ol{\sS}/\PP^1}^{-1}$ containing $0 \in \PP^1$ and choose a section $\vartheta \in H^0(\AA^1, \pi_*K_{\ol{\sS}/\AA^1}^{-1})$. Following Theorem \ref{tm Bottaccin-Markman Poisson str} one can construct a relative Bottacin--Markman Poisson structure $\Theta$ on $\M_{\ol{\sS}/\AA^1}^\hH(\vV)$, 
\begin{align} \label{eq relative Bottacin form}
\Ext^1_{\Oo_{\ol{\sS}_t}}(\Ee, \Ee\otimes K_{\ol{\sS}_t}) \wedge \Ext^1_{\Oo_{\ol{\sS}_t}}(\Ee, \Ee \otimes K_{\ol{\sS}_t})  \stackrel{\circ}{\longrightarrow} & \Ext^2_{\Oo_{\ol{\sS}_t}}(\Ee, \Ee\otimes K_{\ol{\sS}_t}^2) \stackrel{\tr}{\longrightarrow}
\\
\nonumber
& H^{2}(\ol{\sS}_t,K_{\ol{\sS}_t}^2)  \stackrel{\langle \cdot , \vartheta \rangle}{\longrightarrow} H^{2}(\ol{\sS}_t,K_{\ol{\sS}_t}) \cong \cC.
\end{align}
Thanks to \eqref{it sections of omega^-1} of Proposition \ref{pr relative canonical bundle}, we can check that the restriction of $\Theta$ to the subset $\Mm_C(n,d)$ of the central fibre is non-degenerate, hence defines a symplectic form there, as the tangent and the cotangent spaces to the moduli can be identified thanks to a trivialization of $K_{\KK} = \Oo_{\KK}(-2\sigma_\infty)$ over $\Tot(K) = \KK - \{ \sigma_\infty \}$. Furthermore, the canonical bundle is trivial over the generic fibres, $K_S \cong \Oo_S$, so in this case the cotangent space of $\M_{\ol{\sS}/\AA^1}^\hH(\vV)|_t = \M_{S}^{\H}(v)$ is identified with the tangent space. Then, the Poisson form $\Theta_t$ over $t \neq 0$ is non-degenerate and defines a symplectic form which coincides with the Mukai form up to scaling. Hence, the restriction of $\Theta$ to $\mM_{\sS/\AA^1}^\hH(0, n\cC, d-n(g-1))$ defines a relative symplectic form $\varOmega$.
\end{proof}

Let us now consider the case of a K3 surface $X$ containing the smooth curve $C$ of genus $2$. Since $H^1(X,\Oo_X) = 0$ by definition of a K3 surface, $\Pic(X)$ embeds into $H^2(X,\ZZ)$ and one notes that $\Hilb_{\xX/\PP^1}([n\cC])|_{t \neq 0}$ is the linear system $|nC|$ inside $X$. We observe as well, after Lemma \ref{lm description of central fibre} and \eqref{eq identification of Hitchin bases}, that the central fibre of $\Hilb_{\xX/\PP^1}([n\cC])|_{t = 0}$ is the open subset of curves described in the Hitchin base 
\[
B_n \cong |n\sigma_0|_{\supp \cap \sigma_\infty = \emptyset}.
\]

Consider $\pi_* \Oo_{\bar{\aA}}(n\cC)$ over $\PP^1$, which is a trivial vector bundle over $\PP^1 - \{ 0 \}$ of rank $n^2$ after \eqref{eq sections of 2C in A}. The line bundle $\Oo_X(C)$ is big and nef, so one can compute $h^0(X,\Oo_X(C)) = 3$ and $h^0(X,\Oo_X(2C)) = 6$. In Section \ref{sc OGrady} we denoted $\M_\ten = \M_X^\H(0, 2C,- 2)$ as its symplectic resolution provides the ten dimensional O'Grady space $\wt{\M}_\ten$. Denote accordingly $\mM_\ten := \mM_{\xX/\AA^1}^\H(0,nC,- 2)$ and $\bB_\ten := \Hilb_{\xX/\AA^1
}([2\cC])$. One has the following consequence of Theorems \ref{tm Mm_S} and \ref{tm relative symplectic form}.

\begin{corollary} \label {co Mm_ten}
Given a smooth K3 surface $S$ and a smooth projective genus $2$ curve $C$ inside $S$, one can construct $\mM_{\ten}$ and $\bB_\ten$ flat over $\AA^1$, both fitting in the commuting diagram 
\[
\xymatrix{
\mM_{\ten} \ar[rr] \ar[rd] & & \bB_\ten \ar[ld]
\\
& \AA^1. &
}
\]
The horizontal arrow is a surjective fibration, trivial over $\AA^1 - \{ 0 \}$, whose generic fibres coincide with the Beauville--Mukai system \eqref{eq Mukai system},
\[
\M_{\ten} \to |2C|,
\]
and the central fibre at $t = 0$ gives the Hitchin system \eqref{eq Hitchin system},
\[
\Mm_C(2,0) \to H^0(C, K) \oplus H^0(C, K^2).
\]
Furthermore, one can construct a relative symplectic form  on $\mM_\ten \to \AA^1$, which coincides (up to scaling) with Mukai's form on the generic fibres and, over the central fibre, with the standard symplectic form on the moduli space of Higgs bundles. 
\end{corollary}

\section{A non-linear deformation in the $\SL_n$ case}
\label{sc SL_n case}

In this section we study the construction of the non-linear deformation of the reduced Beauville--Mukai system over an abelian variety into the $\SL_n$-Hitchin system treated in \cite{deCataldo&maulik&shen_2} and in \cite{felissetti&mauri}. In these papers, the degeneration is described using the Donagi--Ein--Lazarsfeld degenerations and assuming that the kernel of the Albanese morphism of the generic fibres becomes, at the central fibre, the kernel of the trace and determinant map. The later is stated in \cite{deCataldo&maulik&shen_2, felissetti&mauri} with few details. The biggest contribution of this note is the rigorous proof of the existence of the mentioned degeneration.

For the rest of the article we shall consider $A$ to be the Jacobian of a smooth genus $2$ projective curve $C$, 
\[
A = \Jac^0(C).
\]
Taking $x_0 \in C$ such that $\Oo_C(2x_0) \cong K$, one gets a natural embedding of $C$ into $A$ as a theta divisor and a double cover $C \stackrel{2:1}{\longrightarrow} \PP^1$. The double cover induces an involution $i: C \to C$ which coincides with that given by inverting in $A$ and restricting again to $C$. Hence $\Oo_C(x + i(x)) \cong K$. For every point $a \in A$, we denote by $t_a : A \to A$, the associated traslation in $A$, sending $a' \mapsto a' + a$. 
We know that $\Oo_A(C)$ is ample, and the associated principal polarization
\[
\morph{A}{\widehat{A}}{a}{t_a^*\Oo_A(C) \otimes \Oo_A(C)^{-1}.}{}{\lambda_C}
\]
is an isomorphism, so $A$ is self-dual $A \cong \widehat{A}$. Furthermore, one can compute (see \cite{birkenhake&lange} for instance)
\begin{equation} \label{eq sections of 2C in A}     
h^0(A, \Oo_A(nC)) = n^2.
\end{equation}

Since we recall that the dual determinant is defined in terms of a Fourier--Mukai transform, we focus now on the study of the associated integral functors. We have considered in Section \ref{sc OGrady} the Poincaré bundle $\P \to A \times \widehat{A}$. Let us recall as well the Fourier--Mukai transform \eqref{eq FM in A} and the projections of the first factor $\f : A \times \widehat{A} \to A$ and to the second factor $\g : A \times \widehat{A} \to \widehat{A}$. One has the canonical embedding of $C$ as a theta divisor of $A = \Jac(C)$. Consider  $P_C \to C \times \widehat{A}$ to be the restriction of $\P \to A \times \widehat{A}$ to $C \times \widehat{A}$ and observe that this gives a universal bundle for the classification of ($0$-degree) line bundles on $C$. Taking $\f_C : C \times \widehat{A} \to C$ and $\g_C: C \times \widehat{A} \to \widehat{A}$, it is possible to construct an associated integral functor,
$$
\morph{D^b(C)}{D^b(\widehat{A})}{\Ee^\bullet}{R \g_{C,!}(P_C \otimes \f_{C}^*\Ee^\bullet).}{}{\Phi^{P_C}}
$$

We are interested in the construction of an analog of $\Phi^\P$ replacing $A$ by $\Tot(K)$, where the sheaves classified by $\M_{\six}|_{t = 0}$ are supported. Recall that we denote by $p : \KK \to C$ the structural morphism and set
$$
\P_0 := (p\times \id_{\widehat{A}})^*P_C \to \Tot(K) \times \widehat{A}. 
$$
We consider the obvious projections $\f_{0}:\Tot(K) \times \widehat{A} \to \Tot(K)$ and $\g_{0}:\Tot(K) \times \widehat{A} \to \widehat{A}$. We can consider the associated integral functor (not a derived equivalence)
$$
\morph{D^b(\Tot(K))}{D^b(\widehat{A})}{\Ee^\bullet}{R \g_{0,!}(\P_0 \otimes \f_{0}^*\Ee^\bullet).}{}{\Phi^{\P_0}}
$$

Recall that a coherent sheaf can naturally be seen as a complex in the derived category supported in $0$ degree. One can show the following relation between the transforms of a Higgs bundle and its spectral data.

\begin{lemma} \label{lm identity of FM}
Let $(E,\varphi)$ be a Higgs bundle associated to the spectral data $\Ee$ supported on $C$, where $C \subset \Tot(K)$ is the spectral curve and $\Ee$ is a rank $1$ torsion free sheaf supported on it. If we consider $\Ee$ as a sheaf on $\Tot(K)$, then
$$
\Phi^{\P_0}(\Ee) \cong \Phi^{P_C}(E).
$$
\end{lemma}

\begin{proof}
Consider the commuting diagram
$$
\xymatrix{
 & \Tot(K) \ar[ld]_{\f_0} \ar[rd]^{\g_0} \ar[dd]_{(p \times \id_{\widehat{A}})} &
\\
\Tot(K) \ar[dd]_{p} & & \widehat{A} \ar[dd]^{\id_{\widehat{A}}} 
\\
 & C \times \widehat{A}  \ar[ld]_{\f_C} \ar[rd]^{\g_C} & 
\\
C & & \widehat{A}.
}
$$
Applying the projection formula and base change theorems, one gets
\begin{align*}
\Phi^{\P_0}(\Ee) \cong & R \g_{C,!} R (p \times \id_A)_{!} ((p \times \id_{\widehat{A}})^* P_C \otimes \f_{0}^*\Ee)
\\
\cong & R \g_{C,!} (P_C \otimes R (p \times \id_{\widehat{A}})_{!}\f_{0}^*\Ee)
\\
\cong & R \g_{C,!} (P_C \otimes \f_{C}^*R p_{!}\Ee).
\end{align*}
Finally, observe that, thanks to the spectral correspondence \cite{hitchin_duke, simpson2,BNR}, one has
$$
p_!\Ee \cong (p|_C)_*\Ee \cong E,
$$
and the proposition follows.
\end{proof}

We are working with $A = \Jac^0(C)$ which we recall once again that is self-dual. Note that $P_C \to C \times \widehat{A}$, constructed out of $\P$ and the embedding $C \subset A$, fixes canonically an isomorphism $\widehat{A} \cong A$. Denote by $\widehat{C} \subset \widehat{A}$ the image of $C \subset A$ under this isomorphism. Given a point in $A$ associate to the line bundle $L \to C$, we denote by $\widehat{L} \to A$ the associated point in $\widehat{A}$, note that, by construction, the restriction of $\widehat{L}$ to $C \subset A$ amounts to $L$. The following technical lemma will be crucial for us.

\begin{lemma} \label{lm description of central fibre}
Let $E$ be a vector bundle of rank $n$ and degree $d$ over a smooth projective curve $C$ of genus $2$, we have 
\[
\det_{\widehat{A}} ( \Phi^{P_C}(E)) \otimes \Oo_{\widehat{A}}\left (n\widehat{C} \right )^{-1} \cong \det_C (E) \otimes \Oo_C(d x_0)^{-1}.
\]
\end{lemma}

\begin{proof}
Combining \cite[Th 4.2]{mukai_fourier} and \cite[Prop. 4.3]{mukai_fourier}, one gets, for $\ell \leq 1$, that $\Oo_C(\ell x_0)$ is WIT and $\Phi^{P_C}(\Oo_C(\ell x_0))$ is $\Oo_{\widehat{A}}(\widehat{C})$ in degree $1$. In the case $\ell > 1$, one has that $\Phi^{P_C}(\Oo_C(\ell x_0)) \cong [-1]^*R\Hom_{\Oo_{\widehat{A}}}(F_{2-\ell},\Oo_{\widehat{A}})$. In both cases, one has that
\begin{equation} \label{eq Phi of Oo_C}
\det_{\widehat{A}} (\Phi^{P_C}(\Oo_C(\ell x_0))) \cong \Oo_{\widehat{A}}\left ( \widehat{C} \right ).
\end{equation}

Given $E$, one can always provide a filtration of the form
\begin{equation} \label{eq filtration of E}
0 \subsetneq E_1 \subsetneq \dots \subsetneq E_{n-1} \subsetneq E_n = E,
\end{equation}
where $E_i/E_{i-1} \cong L_i \otimes \Oo_C(\ell_i x_0)$, being $L_i$ a line bundle of trivial degree. It follows naturally that  
\[
\det_C (E) \otimes \Oo_C(d x_0)^{-1} \cong L_1 \otimes \dots \otimes L_n.
\]

Since $\Phi^{P_C}$ preserves distinguished triangles, \eqref{eq filtration of E} implies
\[
\det_{\widehat{A}} (\Phi^{P_C}(E)) \cong \det_{\widehat{A}} (\Phi^{P_C}(L_1 \otimes \Oo_C(\ell_1 x_0))) \otimes \dots \otimes \det_{\widehat{A}} (\Phi^{P_C}(L_n \otimes \Oo_C(\ell_n x_0))).
\]
After \cite[Proposition 3.13 (2)]{mukai_fourier}, the above equality becomes 
\[
\det_{\widehat{A}} \Phi^{P_C}(E) \cong t_{\widehat{L}_1}^* \left ( \det_{\widehat{A}} \Phi^{P_C}(\Oo_C(\ell_1 x_0) \right ) \otimes \dots \otimes t_{\widehat{L}_n}^* \left ( \det_{\widehat{A}} \Phi^{P_C}(\Oo_C(\ell_n x_0) \right ),
\]
and thanks to \eqref{eq Phi of Oo_C},
\[
\det_{\widehat{A}} \Phi^{P_C}(E) \cong t_{\widehat{L}_1}^* \Oo_{\widehat{A}}\left ( \widehat{C} \right ) \otimes \dots \otimes t_{\widehat{L}_n}^* \Oo_{\widehat{A}}\left ( \widehat{C} \right ).
\]
Then, by means of the principal polarization $\lambda_{\widehat{C}} : \widehat{A} \stackrel{\cong}{\to} A$, the proof follows from the above.
\end{proof}

Consider the line bundle $\P_{\PP^1} \to (A \times \PP^1) \times \widehat{A}$ denotes the pull-back of the Poincar\'e bundle $\P \to A \times \widehat{A}$ introduced above. Recall that the normal cone degeneration $\aA$ is given by $\Blow_{C \times \{ 0 \}}(A \times \PP^1) - ( A \times  \{ 0 \})$ and set 
\[
\Pp := (\beta \times_{\PP^1} \id_A)^* \P_{\PP^1}|_{\aA \times \widehat{A}},
\]
where $\beta$ is the structural morphism of the blow-up. Its generic fibre is
\begin{equation} \label{eq generic fibre relative Poincare}
\Pp|_{t \neq 0} \cong \P,
\end{equation}
while at the central fibre,
\begin{equation} \label{eq central fibre relative Poincare}
\Pp|_{t = 0} \cong \P_0.
\end{equation}

Considering the obvious projections
\[
\xymatrix{
 & \aA \times \widehat{A} \ar[rd]^{g} \ar[ld]_{f} &
\\
\aA & & \widehat{A},
}
\]
we define the integral relative functor, 
\begin{equation} \label{eq integral functor}
\morph{D^b_{\PP^1}(\aA)}{D^b(\widehat{A})}{\Ee^\bullet}{Rg_{!}(\Pp \otimes f^*\Ee^\bullet).}{}{\Phi^\Pp}
\end{equation}
Note that, for $t \in \PP^1 - \{ 0 \}$, one has 
\[
\Phi^{\Pp}(\Ee^\bullet|_t) \cong \Phi^{\P}(\Ee|_t),
\]
while for $t = 0$,
\[
\Phi^{\Pp}(\Ee^\bullet|_{t=0}) \cong \Phi^{\P_0}(\Ee|_{t=0}).
\]

Recall the family of divisors $\cC \subset \aA$ and take the associated family $2\cC$ consisting on a family of double curves supported on $\cC$. This comes naturally equipped with the closed immersion $j : \cC \hookrightarrow 2\cC$. Let $E_0$ be a vector bundle with $\det(E_0) \cong \Oo_C(dx_0)$. Recalling that $\cC \cong C \times \PP^1$, consider its pull-back $q^*E_0$ over $\cC$ where $q : \cC \to C$. One can consider the sheaf
\[
\Ee_0 := j_*(q^*E_0),
\]
which is a torsion free sheaf of rank $1$ over $2\cC$. Hence,
\[
\det_\aA (\Ee_0|_t) \cong \Oo_{\aA_t}(2\cC_t),
\]
for every $t \in \PP^1$. Furthermore, after Lemma \eqref{lm description of central fibre}, its dual determinant is
\begin{equation} \label{eq dual determinant of Ee_0}
\det_{\widehat{A}}(\Phi^{\Pp} (\Ee_0 |_t)) \cong \det_{\widehat{A}} \Phi^{P_C} (E_0) \cong \Oo_{\widehat{A}}\left ( n \widehat{C} \right ),
\end{equation}
for all $t \in \PP^1$. Observe that at $t = 0$ we have made use of Lemma \ref{lm identity of FM}.

Using the integral functor \eqref{eq integral functor} and the determinant, we define,
\[
\morph{\mM_{\aA/\AA^1}^\H(0,n\cC, d - n)}{A \times \widehat{A} \times \PP^1}{\Ee}{( \det_{\widehat{A}}(\Phi^\Pp(\Ee)) \otimes \det_{\widehat{A}}(\Phi^\Pp(\Ee_0))^{-1} \, , \, \det_\aA(\Ee) \otimes \det_\aA(\Ee_0)^{-1} ).}{}{\alpha_{\Ee_0}}
\]
One quickly observes that, at the generic fibres over $t \neq 0$, the restriction of the above morphism $\alpha_{\Ee_0}|_{t\neq 0}$ coincides with the Albanese morphism \eqref{eq det and dual det for abelian surfaces}. Let us study the restriction to the central fibre over $t=0$.

\begin{proposition} \label{pr restriction of the Albanese}
The restriction of the Albanese morphism, described in \eqref{eq det and dual det for abelian surfaces}, to the central fibre $\Mm_C(C,d)$ of the Donagi--Ein--Lazarsfeld degeneration $\mM_{\aA/\AA^1}^\hH(0,n\cC, d - n)) \to \AA^1$, becomes the determinant morphism of the underlying vector bundle,
\[
\morph{\Mm_C(n,d)}{A}{(E,\varphi)}{\det(E).}{}{\alpha_{\Ee_0}|_{t=0}}
\]
\end{proposition}

\begin{proof}
The proof follows from the construction of $\alpha_{\Ee_0}$, Lemma \ref{lm description of central fibre} and \eqref{eq the whole slice of Hilb is sent to a point}, which says that all curves in the Hitchin base $B_n$ (understood as a subset of the linear system $|n\sigma_0|$) have determinant equal to $\Oo_\KK(nC)$.
\end{proof}

\begin{remark} \label{rm need to impose restriction on the trace}
We observe from that the Albanese map at the central fibre do not impose any constraint in the set of spectral curves.
\end{remark}

In view of Remark \ref{rm need to impose restriction on the trace}, in order to construct a degeneration of the reduced Beauville--Mukai system, we need to consider a closed subset of the kernel of the Albanese map $\alpha^{-1}_{\Ee_0}(\id_{A}, \Oo_\aA)$.

After \eqref{eq definition of N} and the above discussion, we observe that $\alpha^{-1}_{\Ee_0}(\id_{A}, \Oo_\aA)\to \PP^1$ trivializes over $\PP^1 - \{ 0 \}$, and observe that our generic fibre is $\N^\H_A(0,nC, d-n)$. After Lemma \ref{lm description of central fibre} and \eqref{eq dual determinant of Ee_0}, one has that the central fibre is 
\begin{equation} \label{eq central fibre of Nn'}
\alpha^{-1}_{\Ee_0}(\id_{A}, \Oo_\aA)|_{t=0} \cong \Mm_C(n,d)^{\det = \Oo_C(d x_0)},
\end{equation}
the locus of the Higgs moduli space of fixed determinant, but no restriction is imposed in the trace of the Higgs field.

Consider $\pi_* \Oo_{\bar{\aA}}(n\cC)$ over $\PP^1$, which is a trivial vector bundle over $\PP^1 - \{ 0 \}$ of rank $n^2$ after \eqref{eq sections of 2C in A}. By upper semicontinuity of cohomology, the coherent sheaf $\pi_* \Oo_{\bar{\aA}}(2\cC)$ has torsion only at $t = 0$, killing the torsion, one may take the rank $n^2$ vector bundle $\mathbf{V} \subset \pi_* \Oo_{\bar{\aA}}(n\cC)$ over $\PP^1$. Projectivizing $\mathbf{V}\to \PP^1$, one gets the $\PP^{(n^2-1)}$-bundle $\ol{\Vv} \to \PP^1$. The points of $\ol{\Vv}|_t$ represent $\CC^*$-rays of sections of $\Oo_{\bar{\aA}_t}(\cC_t)$, which are completely determined by their vanishing locus. This provides the inclusion 
\[
\ol{\Vv} \subset \Hilb_{\ol\aA/\PP^1}([n\cC]), 
\]
and note that generic fibres over $t\neq 0$ are identified with the linear system $|nC|$ in $A$. We may consider the restriction to those curves contained in $\aA \subset \ol\aA$,
\[
\Vv := \ol{\Vv} \cap \Hilb_{\aA/\PP^1}([n\cC]).
\]

We now describe the central fibre of $\Vv$.

\begin{lemma} \label{eq Vv_0 is the reduced Hitchin base}
Under the identification \eqref{eq identification of support bases} and \eqref{eq identification of Hitchin bases}, one has that
\begin{equation} \label{eq description of Vv_0}
\Vv|_{t=0} \cong \bigoplus_{i = 2}^n H^0(C, K^i),
\end{equation}
coinciding with the base of the reduced Hitchin system \eqref{eq reduced Hitchin system}.
\end{lemma}

\begin{proof}
We recall from \eqref{eq identification of support bases} and \eqref{eq identification of linearsystem with deformations} that the central fibre of $\Hilb_{\aA/\PP^1}([n\cC])$ corresponds to the space of linear deformations of the multiple curve $nC$ inside $\KK$, 
\[
\Hilb_{\aA/\PP^1}([n\cC]) |_{t=0} \cong |n\sigma_0|_{\supp \cap \sigma_\infty = \emptyset} = H^0(n\sigma_0,\Oo_\KK(n\sigma_0)|_{n\sigma_0}).
\]
Note that, since the normal bundle of $C$ inside $A$ coincides with the normal bundle of $\sigma_0 \cong C$ inside $\KK$, both are described by the same ribbon that we shall denote by $nC$. Hence, we can identified the linear deformations for each of the fibres of $n\cC$. The subspace $\Vv|_{t = 0} \subset \Hilb_{\aA/\PP^1}([n\cC]) |_{t = 0}$ is given by linear deformations of the curve induced by those that preserve the determinant within $A$, that is, those linear deformations describing an closed subset of the linear system $|nC|$ in $A$, the fibre of the map \eqref{eq det_Ss}. By \eqref{eq sections of 2C in A}, one has that the dimension of $\Vv|_{t = 0}$ is $n^2-1$.

The linear deformations of $nC$ in $A$ are classified by the sections of $\Oo_C(nC) \otimes ( \Oo_C \oplus (\Ii_C/\Ii_C^2) \oplus \dots \oplus (\Ii_C^{n-1}/\Ii_C^n))$, where $(\Ii_C/\Ii_C^2) \cong \Oo_C(-C) \cong K^{-1}$ as we lie on an abelian surface. Let us denote by $W$ the subspace of linear deformations classified by sections of $\Oo_C(nC) \otimes (\Ii_C^{n-1}/\Ii_C^n) \cong K$. The elements of $W$ vanish on the open subset $(n-1)C$ so they correspond to those deformations deforming the multiple curve $nC$ into a multiple curve $nC'$ with support $C'$ different from $C$. All these linear deformations should contain those given by the differentials of $t_a : A \to A$, defined using the group structure. The former are classified by $T_a A \cong H^1(A,\Oo_A) \cong H^0(C,K)$, whose dimension coincides with that of $W$ so both are equal. We then know that no deformation classified by $W$ preserves the determinant $\Oo_C(nC)$. Note that one can project the space of all deformations, $\Hilb_{\aA/\PP^1}([n\cC]) |_{t = 0}$, into our subspace $W$. Under this projection, the subspace $\Vv|_{t = 0}$ is sent to $0$ as these deformations preserve the determinant. Hence, $\Vv|_{t = 0}$ is a subspace of $\bigoplus_{i=2}^n H^0(C,K^i)$, and both coincide as they have the same dimension.
\end{proof}

According with \eqref{eq central fibre of Nn'} and Lemma \ref{eq Vv_0 is the reduced Hitchin base}, we define 
\[
\nN_{\sS/\PP^1}^\H(0,n\cC, d - n) := \alpha^{-1}_{\Ee_0}(\id_{A}, \Oo_\aA)|_{\Vv},
\]
and, recalling the symplectic form $\varOmega$ defined in Theorem \ref{tm relative symplectic form}, set
\[
\bar{\varOmega} := \varOmega |_{\nN}.
\]
We observe that $\bar{\varOmega}$ provides, fibrewise, a relative symplectic structure on $\nN_{\sS/\PP^1}^\H(0,n\cC, d - n)$. This leads to the main result of this section.

\begin{theorem} \label{tm Nn_A}
Given a smooth projective curve $C$ of genus $g = 2$ and consider its Jacobian $A = \Jac^0(C)$. One can construct $\nN_{\sS/\PP^1}^\H(0,n\cC, d - n)$ and $\Vv$ flat over $\PP^1$, both fitting in the commuting diagram 
\begin{equation} \label{eq commuting diagram Nn_A}
\xymatrix{
\nN_{\sS/\AA^1}^\H(0,n\cC, d - n) \ar[rr] \ar[rd] & & \Vv \ar[ld]
\\
& \PP^1. &
}
\end{equation}
The horizontal arrow is a surjective fibration, trivial over $\PP^1 - \{ 0 \}$, and the generic fibres of $\nN_{\sS/\PP^1}^\H(0,n\cC, d - n)$ coincide with the reduced Beauville--Mukai system
\[
\N_A^\H(0, nC, d-n) \to |nC|, 
\]
while the central fibre at $t = 0$ gives the $\SL_n$-Hitchin system \eqref{eq reduced Hitchin system},
\[
\N_C(n,d) \to \bigoplus_{i = 2}^n H^0(C, K^i).
\]
Furthermore, one can construct a relative symplectic form $\bar{\varOmega}$ on $\nN_{\sS/\AA^1}^\H(0,n\cC, d - n) \to \AA^1$, which coincides (up to scalling) with the reduced Mukai's form $\ol{\Omega}$ on the generic fibres and, over the central fibre, with the standard reduced symplectic form on the moduli space of Higgs bundles $\ol{\Omega}_0$.
\end{theorem}

\begin{proof}
Everything follows easily from Theorems \ref{tm Mm_S} and \ref{tm relative symplectic form}, and Lemmas \ref{lm identity of FM}, \ref{lm description of central fibre} and \ref{eq Vv_0 is the reduced Hitchin base}.
\end{proof}

Recall that, in Section \ref{sc OGrady}, we denoted $\N_\six = \N^\H_A(0,2C, -2)$, as the symplectic resolution of this moduli space is the six dimensional O'Grady space $\widetilde{\N}_{\six}$, what explains the subindex used in the notation. We shall use the same subindex to denote $\nN_\six := \nN_{\sS/\AA^1}^\H(0,2\cC, - 2)$ and $\Vv_\six$.

\begin{corollary} \label{co Nn_six}
Given a smooth projective genus $2$ curve $C$ and taking the abelian surface $A = \Jac^0(C)$, one can construct $\N_{\six}$ and $\Vv_\six$ flat over $\PP^1$, both fitting in the commuting diagram 
\[
\xymatrix{
\nN_{\six} \ar[rr] \ar[rd] & & \Vv_\six \ar[ld]
\\
& \AA^1. &
} 
\]
The horizontal arrow is a surjective fibration, trivial over $\AA^1 - \{ 0 \}$, whose generic fibres coincide with the reduced Beauville--Mukai system \eqref{eq Mukai system},
\[
\N_{\six} \to |2C|,
\]
and the central fibre at $t = 0$ gives the $\SL_2$-Hitchin system \eqref{eq Hitchin system},
\[
\N_C(2,0) \to H^0(C, K^2).
\]
Furthermore, one can construct a relative symplectic form $\varOmega_\six$ on $\nN_\six \to \AA^1$, which coincides (up to scalling) with the reduced Mukai's form $\ol{\Omega}$ on the generic fibres and, over the central fibre, with the standard reduced symplectic form on the moduli space of Higgs bundles $\ol{\Omega}_0$. 
\end{corollary}

\section{Degenerating O'Grady spaces into symplectic resolutions of moduli spaces of Higgs bundles}
\label{sc main result}

In this last section we address the construction of a non-linear deformation of O'Grady ten and six dimensional spaces into the symplectic resolution of, respectively, the Higgs moduli spaces $\M_C(2,0)$ and $\N_C(2,0)$ where the genus of the curve is $g=2$. As described in the Appendix of \cite{felissetti&mauri}, and as can be infered from Theorems \ref{tm resolution for surfaces}, \ref{tm resolution for abelian surfaces} and \ref{tm desingularization}, such degeneration is obtained by blowing up the singular locus relative to the projection to $\PP^1$.

Then, we shall first identify the singular locus of the fibres of our degenerating families, $\mM_\ten$ and $\nN_\six$, in order to perform a relative blow-up over these loci, giving rise to a relative symplectic resolution of singularities.

Pick a K3 surface $X$ containing a genus $2$ curve and consider the abelian surface $A$ obtained as the Jacobian of the genus $2$ curve $C$. We consider $\mM^\hH_{\xX/\PP^1}(0, \cC, -1)$ and $\mM^\hH_{\aA/\PP^1}(0, \cC, -1)$. Then, observe that the strictly semistable loci of the fibres of $\mM_{\ten} \to \PP^1$ and $\nN_{\six} \to \PP^1$ is given by
\[
\mM_{\ten}^{\sps} := \left \lbrace \Ee_1 \oplus \Ee_2  \textnormal{ where } \Ee_i \in \mM^\hH_{\xX/\PP^1}(0, \cC, -1) \right \rbrace 
\]
and 
\[
\nN_{\six}^{\sps} := \left \lbrace \Ee_1 \oplus \Ee_2 \textnormal{ where } \Ee_i \in \mM^\hH_{\aA/\PP^1}(0, \cC, -1) \right \rbrace \cap \nN_{\six}.
\]

Finally, let us consider 
$$
\wt{\mM}_{\ten} := \Blow_{\mM_{\ten}^{\sps}}(\mM_{\ten})
$$
and
$$
\wt{\nN}_{\six} := \Blow_{\nN_{\six}^{\sps}}(\nN_{\six}),
$$
and take the composition of the blow-up projections $\wt{\mM}_{\ten} \to \mM_{\ten}$ and $\wt{\nN}_{\six} \to \nN_{\six}$ with the structural morphisms $\mM_{\ten} \to \PP^1$ and $\nN_{\six} \to \PP^1$.

We finally arrive to the main result of these notes. Note that the construction of these degenerations has been described in the Appendix of \cite{felissetti&mauri}, although the degeneration of the holomorphic symplecic structures is not addressed there. This analysis appears in the theorem below, by means of Theorem \ref{tm relative symplectic form} taken from \cite{franco}. 

\begin{theorem}
Let $C$ be a genus $2$ curve lying inside a K3 surface $S$. Then, the relative scheme $\wt{\mM}_{\ten} \to \PP^1$ is a non-linear deformation of symplectic varieties, from the ten dimensional O'Grady space $\wt{\M}_{\ten}$ into the symplectic resolution $\wt{\Mm}_C(2,0)$ of the Higgs moduli space $\Mm_C(2,0)$.

Also, the relative scheme $\wt{\nN}_{\six} \to \PP^1$ is a non-linear degeneration of symplectic varieties, from the six dimensional O'Grady space $\wt{\N}_{\six}$ into the symplectic resolution $\wt{\Nn}_C(2,0)$ of the moduli space $\Nn_C(2,0)$ of Higgs bundles with $0$ trace and trivial determinant.
\end{theorem}

\begin{proof}
Since both $\wt{\mM}_{\ten}$ and $\wt{\nN}_{\six}$ are constructed by blowing-up the strictly polystable loci of, respectively, $\mM_{\ten}$ and $\nN_{\six}$, the proof follows after Theorems \ref{tm resolution for surfaces}, \ref{tm resolution for abelian surfaces} and \ref{tm desingularization}, and Corollaries \ref{co Mm_ten} and \ref{co Nn_six}.
\end{proof}


\begin{thebibliography}{99999}
\bibliographystyle{plain}






\bibitem[Be1]{beauville_IHSM} A. Beauville, {\it Vari\'et\'es k\"ahleriennes dont la premi\`ere classe de Chern est nulle}. J. Differential Geom., {\bf 18} (1983) 755--782. 


\bibitem[Be2]{beauville_book} A. Beauville, {\it Complex algebraic surfaces}, Cambridge University Press, 2nd edition (1996).

\bibitem[Be3]{beauville_fibr} Beauville, {\it Syst\`emes hamiltoniens compl\`etement int\'egrables associ\'es aux surfaces K3}, Symposia Mathematica (1991).

\bibitem[Be4]{beauville_sing} A. Beauville, {\it Symplectic singularities}, Invent. Math. 139 (2000), 541--549.

\bibitem[Be5]{beauville_3} Beauville, {\it Holomorphic symplectic geometry: a problem list}, Complex and Differential geometry, Springer (2011).


\bibitem[BNR]{BNR}
A. Beauville, M. S. Narasimhan and S. Ramanan,
Spectral curves and the generalised theta divisor.
J. Reigne Angew. Math. {\bf 398} (1989), 169--179.


\bibitem[BS]{bellamy&schedler} Bellamy and Schedler, {\it On the (non)existence of symplectic resolutions of linear quotients}, Math. Res. Lett. (2016).

\bibitem[BR]{biswas&ramanan} I. Biswas and S. Ramanan, {\it An infinitesimal study of the moduli pf Hitchin pairs}, J. London Math. Soc. (2) 49 (1994) 219--231.


\bibitem[BL]{birkenhake&lange} C. Birkenhake and H. Lange, {\it Complex Abelian Varieties}, Springer-Verlag Berlin Heidelberg, 2nd edition (2004).


\bibitem[BLR]{neron} S. Bosch, W. Lutkebohmert and M. Raynaud, Neron models,
{Ergebnisse der Mathematik und ihrer Grenzgebiete (3)},
\textbf{21}. Springer-Verlag, Berlin, 1990.


\bibitem[Bo1]{bottacin_1} F. Bottacin, {\it Poisson structures on moduli spaces of sheaves over Poisson surfaces}, Invent. math. 121,421-436 (1995).

\bibitem[Bo2]{bottacin_2} F. Bottacin, {\it Symplectic geometry on moduli spaces of stable pairs}, Ann. Sci. É.N.S., IV 28 (4) (1995), 391--433.

\bibitem[dC]{deCataldo}
M. A. de Cataldo.
A support theorem for the Hitchin fibration: the case of $SL_n$.
Comp. Math. 153 (6) (2017) 1316--1347.

\bibitem[dCHM]{deCataldo&hausel&migliorini} M.A. de Cataldo, T. Hausel, and L. Migliorini. {\it Exchange between perverse and weight filtration for the Hilbert schemes of points of two surfaces}, J. Singul., 7:23--38, (2013).

\bibitem[dCMS1]{deCataldo&maulik&shen_1} M. A. de Cataldo, D. Maulik and J. Shen, {\it Hitchin fibrations, abelian surfaces and the P=W conjecture}, J. Amer. Math. Soc. 35 (2022), 911--953 .


\bibitem[dCMS2]{deCataldo&maulik&shen_2} M. A. de Cataldo, D. Maulik and J. Shen, {\it On the P=W conjecture for $\SL_n$}, to appear in Selecta Math.


\bibitem[Co]{corlette} K. Corlette, Flat G-bundles with canonical metrics, \textit{J.
Diff. Geom.}, 28(3), 361--382, 1988.


\bibitem[DEL]{donagi&ein&lazarsfeld} R. Donagi, Ein and Lazarsfeld, {\it Nilpotent cones and sheaves on K3 surfaces}, Contemp. Math. (1997)




\bibitem[Do]{donaldson} S. Donaldson, Twisted harmonic maps and the self-duality
equations, \textit{Proc. London Math. Soc.} (3), 55(1):127--131, 1987.

\bibitem[FM]{felissetti&mauri} C. Felisetti and M. Mauri, {\it P=W conjectures for character varieties with symplectic resolution} Journal de l'École polytechnique, 2022, pp. 52.

\bibitem[Fr]{franco} E. Franco, {\it Degeneration of natural Lagrangians and Prymian integrable systems}, Math. Zeit., to appear.


\bibitem[Fu]{fulton} W. Fulton, {\it Intersection theory},
Springer-Verlag 2nd ed. (1998), Berlin, New York.

\bibitem[Gi]{gieseker} D. Gieseker, {\it On the moduli of vector bundles on an algebraic surface}, Annals Math., 106 (1977), 45--60

\bibitem[Ha]{hartshorne}
{R. Hartshorne}, {\it Algebraic Geometry}, Graduate Texts in Mathematics
\textbf{52}, Springer-Verlag, 1977.



\bibitem[Hi1]{hitchin_duke} Hitchin, {\it Stable bundles and integrable systems}, Duke Math. J. (1987).



\bibitem[Hi2]{hitchin-self}
N. J. Hitchin,
The self-duality equations on a Riemann surface.
Proc. London Math. Soc.  {\bf 55} (1987), 59--126.



\bibitem[Hi3]{hitchin_char}
N. J. Hitchin, 
Higgs bundles and characteristic classes.
In \emph{Arbeitstagung Bonn 2013; In Memory of Friedrich Hirzebruch},
Progr. Math.  Birkh\"auser {\bf 319},  2016,  247--264.


\bibitem[Hu1]{Huybrechts2}
D.~Huybrechts
{\it Compact hyperk\"ahler manifolds: basic results},
Invent. Math. {\bf 135} (1999), 63--113.

\bibitem[Hu2]{lecture-on-K3}
D.~Huybrecths,
{\it Lectures on K3 surfaces}.

\bibitem[Hu3]{Huybrechts}
D.~Huybrecths,
{\it Moduli spaces of hyperk\"ahler manifolds and Mirror symmetry}.
Intersection theory and moduli,
ICTP Lect. Notes, XIX, pages 185-247. Abdus Salam Int. Cent.
Theoret. Phys., Trieste, 2004.

\bibitem[Hu4]{huybrechts_CY} Huybrechts, {\it Compact hyperk\"ahler manifolds}, in Calabi-Yau manifolds and related geometries, Springer (2003). 


\bibitem[HL]{huybrechts&lehn}
D. Huybrechts and M. Lehn,
{\it The geometry of moduli spaces of sheaves},
Cambridge.


\bibitem[KL]{KL} Kaledin and Sorger, {\it Local structure of hyperk\"aehler singularities in O'Grady's spaces}, Moscow Math. J. (2007).

\bibitem[KLS]{kaledin&lehn&sorger} Kaledin, Lehn and Sorger, {\it Singular symplectic moduli spaces}, Invent. Math. (2006).


\bibitem[KY]{kiem&yoo} Y-H. Kiem and S-B. Yoo, {\it The stringy E-function of the moduli space of Higgs bundles with trivial determinant}, Math. Nach. (2008).

\bibitem[Ki]{kiem} Y-H. Kiem, {\it Y-H. Kiem, On the existence of a symplectic desingularization of some moduli spaces of sheaves on
a K3 surface}, Compos. Math. 141 (2005), 902--906.

\bibitem[Kw]{kirwan} Kirwan, {\it Partial desingularizations of quotients of nonsingular varieties and their Betti numbers}, Ann. Math. (1985).

\bibitem[LS]{lehn&sorger} Lehn and Sorger, {\it La singularit\'e de O'Grady}, J. Alg. Geom. (2006).



\bibitem[Ma]{matsushita} D. Matsushita, {\it On fibre space structures of a projective irreducible
symplectic manifold}, Topology 38 (1) (1999), 79--83. {\it Addendum}, Topology 40 (2) (2001), 431--432.


\bibitem[Mr1]{markman} E. Markman, {\it Spectral curves and integrable systems}, Comp. Math., 93 (3) (1994), 255--290.

\bibitem[Mr2]{Markman0}
E.~Markman,
{\it Integral constraints on the monodromy group of the hyperk\"ahler resolution of a symmetric product of a K3 surface}.
Internat. J. Math. {\bf 21} (2010), 169--223.

\bibitem[Mr3]{Markman}
E.~Markman,
{\it A survey of Torelli and monodromy results for holomorphic-symplectic varieties},
Complex and differential geometry, 257--322, Springer Proc. Math., 8, Springer, Heidelberg 2011.

\bibitem[MT]{MT} Markushevich and Tikhomirov, {\it New symp. $V$-manifolds of dim. $4$ via relative compactified Prymian}, Int. J. Math. (2007).

\bibitem[Mu1]{mukai_fourier} S. Mukai, {\it Duality between $D(X)$ and $D(\hat{X})$ with its application to Picard sheaves}, Nagoya Math. J. 81 (1981) 153--175. 

\bibitem[Mu2]{mukai1}
S. Mukai,
{\it Symplectic structure on the module space of sheaves on an abelian or K3 surface}. 
Invent. Math. {\bf 77} (1984), 101--116.

\bibitem[Mu3]{mukai2}
S. Mukai,
{\it On the moduli space of bundles on K3 surfaces. I}. 
In {\it Vector bundles on algebraic varieties} (Bombay, 1984), Tata Institute of Fundamental Research
Studies in Mathematics, vol. 11, Oxford University Press, Bombay, 1987, pp. 341--413.


\bibitem[Ni]{nitsure}
N. Nitsure,
Moduli of semistable pairs on a curve. 
Proc. London Math. Soc. {\bf 62} (1991), 275--300.


\bibitem[Or]{orlov}
D. O. Orlov,
{\it Equivalences of derived categories and K3 surfaces}.
J. Math. Sci. (New York), 84 (1997), pp. 1361--1381.


\bibitem[OG1]{OGrady92} O'Grady, {\it Donaldson's polynomials for K3 surfaces}, J. Diff. Geom. (1992).


\bibitem[OG2]{OGrady97} O'Grady, {\it The weight-two Hodge structure of moduli spaces of sheaves on a K3 surface}, J. of Alg. Geom. (1997).


\bibitem[OG3]{OGrady_1} O'Grady, {\it Desingularized moduli spaces of sheaves on a K3 surface}, J. reine angew Math. (1999).

\bibitem[OG4]{OGrady_2} O'Grady, {\it A new six dimensional irreducible symplectic variety}, J. Alg. Geom. (2003).


\bibitem[OG5]{OGrady_4} O'Grady, {\it Hyper\"ahler manifolds and algebraic geometry}, European Congress of Mathematics Eur. Math. Soc. (2005).

\bibitem[PR]{perego&rapagnetta} A. Perego and A. Rapagnetta, {\it Deformation of the O’Grady moduli space}, J. reine angew Math. 678 (2013), 1--34.


\bibitem[Ra]{rapagnetta} A. Rapagnetta, {\it Topological invariants of O’Grady’s six dimensional irreducible symplectic variety},
Math. Z. 256 (2007), 1--34.


\bibitem[Sb]{schaub} D. Schaub.
Courbes spectrales et compactifications de Jacobiennes.
{Mathematische Zeitschrift} 227, issue 2  (1998) 295--312.


\bibitem[Sw]{sawon} J. Sawon, {\it Lagrangian fibrations by Prym varieties}, Matem\'atica Contempor\^anea, 47, 182--227.

\bibitem[SS]{sawon&shen} J. Sawon and C. Shen, {\it Deformations of compact Prym fibrations to Hitchin systems}, Bull. London Math. Soc. (2022), DOI:10.1112/blms.12643.



\bibitem[Si1]{simpson1} 
C.T. Simpson. 
Moduli of representations of the fundamental group of a smooth projective variety I. 
Publ. Math., Inst. Hautes Etud. Sci. {\bf 79} (1994), 47--129.

\bibitem[Si2]{simpson2} 
C.T. Simpson.
Moduli of representations of the fundamental group of a smooth projective variety II.
Publ. Math., Inst. Hautes Etud. Sci. {\bf 80} (1995), 5--79.


\bibitem[Ti]{tirelli} Tirelli, {\it Symplectic resolutions for Higgs moduli spaces}, aXv: 1701.07468.

\bibitem[Y1]{yau1} S. -T. Yau, {\it Calabi's conjecture and some new results in algebraic geometry},  Proc. Nat. Ac. Sci. Am., 74 (5) (1977) 1798--1799.

\bibitem[Y2]{yau2} S. -T. Yau, {\it On the Ricci curvature of a compact Kähler manifold and the complex Monge-Amp\`{e}re equation. I}, Comm. Pure App. Math., 31 (3) (1978) 339--411.

\bibitem[Yo]{yoo} S. -B. Yoo, {\it A desingularization of the moduli space of rank 2 Higgs bundles over a curve},  Taiwanese J. Math. 25(2) (2021) 257--301.

\bibitem[Ys1]{yoshioka_1} K. Yoshioka, {\it Some notes on the moduli of stable sheaves on elliptic surfaces}, Nagoya Math. J. 154 (1999), 73--102.

\bibitem[Ys2]{yoshioka_2} K. Yoshioka, {\it Moduli spaces of stable sheaves on abelian surfaces}, Math. Ann. 321 (2001) 817--884.


\end{thebibliography}
\end{document}